\documentclass{article}

\usepackage{arxiv}

\usepackage[utf8]{inputenc} 
\usepackage[T1]{fontenc}    
\usepackage{hyperref}       
\usepackage{url}            
\usepackage{booktabs}       
\usepackage{amsfonts}       
\usepackage{nicefrac}       
\usepackage{microtype}      
\usepackage{lipsum}

\usepackage[all]{xy,xypic}
\usepackage{amsfonts,amssymb,amsmath,amsgen,amsopn,amsbsy,theorem,graphicx,epsfig}
\usepackage{eufrak,amscd,bezier,latexsym,mathrsfs,enumerate,multirow}
\usepackage[utf8]{inputenc}\usepackage[english]{babel}
\usepackage[dvipsnames]{xcolor}
\usepackage[pagewise]{lineno}

\usepackage{multirow}
\usepackage{colortbl}
\usepackage{tikz}
\usetikzlibrary{arrows}
\usepackage[ruled,vlined]{algorithm2e}
\renewcommand{\arraystretch}{1.5}
\usepackage{float}
\restylefloat{table}

\newtheorem{theorem}{Theorem}[section]

\newtheorem{definition}[theorem]{Definition}
\newtheorem{example}[theorem]{Example}

\newtheorem{lemma}[theorem]{Lemma}
\newtheorem{notation}[theorem]{Notation}

\newtheorem{proposition}[theorem]{Proposition}
\newtheorem{remark}[theorem]{Remark}

\renewcommand{\phi}{\varphi}

\title{A new approach for computing the distance and the diameter in circulant graphs}

\author{
  Laila Loudiki\\
  Department of Mathematics and Computer Science\\
  Polydisciplinary Faculty of Safi \\
  Cadi Ayyad University\\
  Safi, Morocco \\
  \texttt{laila.loudiki@ced.uca.ma} \\
   \And
 Mustapha Kchikech \\
  Department of Mathematics and Computer Science\\
  Polydisciplinary Faculty of Safi \\
  Cadi Ayyad University\\
  Safi, Morocco \\
  \texttt{m.kchikech@uca.ac.ma} \\
   \AND
   El Hassan Essaky \\
     Department of Mathematics and Computer Science\\
     Polydisciplinary Faculty of Safi \\
  Cadi Ayyad University\\
  Safi, Morocco \\
  \texttt{essaky@uca.ma}
}

\begin{document}
\maketitle

\begin{abstract}
The diameter of a graph is the maximum distance among all pairs of vertices. Thus a graph $G$ has diameter $d$ if any two vertices are at distance at most $d$ and there are two vertices at distance $d$. We are interested in studying the diameter of circulant graphs $C_n(1,s)$, i.e., graphs with the set $\{0,1,\ldots, n-1\}$ of integers as vertex set and in which two distinct vertices $i,j \in \{0,1,\ldots, n-1\}$ are adjacent if and only if $|i-j|_n\in \{1,s\}$, where $2\leq s\leq \lfloor \frac{n-1}{2} \rfloor$ and $|x|_n=\min(|x|, n-|x|)$. Despite the regularity of circulant graphs, it is difficult to evaluate several parameters, in particular the distance and the diameter. To the best of our knowledge, there is no formulas providing exact values for the distance and the diameter of $C_n(1,s)$ for all $n$ and $s$.  In this context, we present in this paper a new approach, based on a  simple algorithm, that gives  exact values for the distance and the diameter of circulant graphs.
\end{abstract}


\section{Introduction}
\label{Intro}

Circulant graphs form an important and very well-studied class of graphs \cite{monakhova2012survey}. They find applications to the computer network design, telecommunication networking, distributed computation, parallel processing architectures, and VLSI design.

For $n\in \mathbb{N}$ with $n\geq 4$, let $S=\{s_1, s_2, \ldots, s_k \}$ where $s_i$ $(i=1,2, \ldots, k)$ are positive integers such that $1\leq s_1 < s_2 < \ldots < s_k \leq \lfloor \frac{n}{2}\rfloor$. The \textit{circulant graph} $C_n(S)=(V,E)$ has the set $V=\{0, 1, \ldots, n-1 \}$ of integers as a vertex set and in which two distinct vertices $i,j\in \{0, 1, \ldots, n-1 \}$ are adjacent if and only if $|i-j|_n\in S$,  where $|x|_n=\min(|x|, n-|x|)$.

The elements of the generating set $S = (s_1, s_2, \ldots , s_k)$ are called generators (chords). A parametric description of the form $(n; S)$ completely specifies the circulant of order $n$ and dimension $k.$ Since circulants belong to the family of Cayley graphs, any undirected circulant graph $C_n(S)$ is vertex-transitive and $|S|$-regular.  Circulant graphs are also known as starpolygon graphs \cite{boesch1984circulants}, cyclic graphs \cite{david1972enumeration}, distributed loop networks \cite{bermond1995distributed}, chordal rings \cite{barriere2000fault}, multiple fixed step graphs \cite{fabrega1997fault}, point-symmetric graphs \cite{turner1967point}, in Russian as Diophantine structures \cite{monakhova1979synthesis}.

The bigger is the cardinal of $S$, the smaller is the value of the diameter of  $C_n(S)$. For example, when $S = \{1, 2, \ldots , \lfloor \frac{n}{2}\rfloor\}$, the circulant graph $C_n(S)$ is isomorphic to the complete graph $K_n$ of order $n$ and diameter $diam(K_n)=1$. Therefore, we consider, in this paper, the connection set $S=\{1, s\}$ where $s$ is an integer such that  $2\leq s\leq \lfloor \frac{n-1}{2} \rfloor,$ and we focus on $4$-regular circulant graph, $C_n(1,s)$, on $n$ vertices with respect to $S.$ It can be defined as a graph on vertex-set $V=\{i  :  i\in \mathbb{Z}_n\}$ and edge-set $E=\{(i,i\pm r) :  r\in \{1,s\}\}$. Since $gcd(n,1,s)=1$, the graph $C_n(1,s)$ is connected. The distance $d(i,j)$ between two vertices $i$ and $j$ in $C_n(1,s)$ is the length of a shortest path joining $i$ and $j$. The diameter of $C_n(1,s)$, denoted $diam(C_n(1,s))$, is the maximum distance among all pairs of vertices in $C_n(1,s)$. Here we are interested in the following problem. 
 
 \textit{\textbf{Problem.}} Given $n,s$, determine $diam(C_n(1,s))$.
 
The problem has applications in the design of interconnection networks.   The exact calculation of the
diameter of circulant graphs is a well-studied problem even for the case $|S|=2$. For theoretical results, upper and lower bounds were giving  (see \cite{monakhova2012survey, chen2005diameter}). As for algorithmic  results, this problem was first stated by Wong and Coppersmith who gave a heuristic algorithm for solving it \cite{wong1974combinatorial}. It was also studied by Zerovnik and Pisanski \cite{zerovnik1993computing}  who established an algorithm for computing the diameter of $C_n(1,s)$ with running time $O(log(n))$. However, there was no formulas giving exact values for the diameter of $C_n(1,s)$ for all $n$ and $s$.   Our approach which makes it possible to determine the distances in a circulant, allowed us to give exact formulas of the diameter of $C_n(1,s)$ for various values of $n$ and $s$.

 The rest of this paper is structured as follows. In section \ref{Algo}, we focus on studying paths joining any two vertices in $C_n(1,s)$, and we determine the equivalence class of paths existing in circulant graphs. Later, we present a formula that provides the exact value for the distance between any two vertices in $C_n(1,s)$, then, for given $n$ and $s$, we reveal an algorithm for computing the diameter of $C_n(1,s)$. For theoretical results, section \ref{diam} provides exact formulas for the diameter of circulant graphs $C_n(1,s)$ for almost all $n$ and $s$, and  gives an upper bound for the rest of the values of $n$ and $s$ (see table below where $\lambda=\frac{n}{s},$ $\gamma$ is the reminder of dividing $n$ by $s,$ $a=\frac{s}{\gamma},$ $b$ is the reminder of dividing $s$ by $\gamma,$ $p_0=\lfloor \frac{\lambda+\gamma}{2} \rfloor$, $p_1=\lfloor \frac{\gamma-b+(a+1)\lambda+1}{2} \rfloor$, $p_2=\lfloor \frac{\gamma+b+(a-1)\lambda+1}{2} \rfloor$, $p_3=\lfloor \frac{b+a\lambda+1}{2} \rfloor,$ and $e_1= \min\{\max\{p_1, p_3\},\max\{p_0, p_2\}\}$).

\begin{table}[h!]
\medskip
\centering\renewcommand{\arraystretch}{1.2}
  \begin{tabular}{|c|c|c|l|}
   \cline{4-4}
  \multicolumn{2}{c}{} & & $\centering$ $\centering$ $\centering$ $\centering$ $\centering$ $\centering$ $\centering$ $\centering$ $\centering$ $\centering$ $\centering$ $\centering$ $\centering$ $\centering$ $\centering$ $\centering$ $\centering$  $\centering$ $\centering$ $\centering$ $\centering$ $\centering$ $\centering$ Diameter of $C_n(1,s)$  \\ 
  \hline 
  \multicolumn{2}{|c}{$\centering$ $\centering$ $\centering$ $\centering$ $\centering$ $\centering$ $\centering$ $\centering$ $\centering$ $\centering$ $\centering$ $\centering$ $\centering$ $\centering$ $\centering$ $\centering$ $\centering$ $\centering$  $\gamma =0$} & & $=\lfloor \frac{\lambda+s-1}{2}\rfloor$ \\ 
  \hline 
 \multirow{7}{*}{\rotatebox{90}{\centering $\lambda\geq\gamma$}}& \multirow{2}{*}{$n$ even}&$s$ odd&$=\lceil \frac{\lambda}{2} \rceil + \frac{s-1}{2} -( \min(\lceil \frac{\gamma}{2}\rceil,\lceil \frac{s-\gamma+1}{2}\rceil) -1)$\\
    \cline{3-4}
                       &  &$s$ even&$=\begin{cases} 
      \lceil \frac{\lambda}{2} \rceil + \frac{s-\gamma}{2}  & \mbox{if   $\gamma \leq 2\lceil \frac{s-2}{4} \rceil$,} \\ 
     \lfloor \frac{\lambda}{2} \rfloor +\frac{\gamma}{2} &  \mbox{otherwise.} 
 \end{cases}$\\
    \cline{2-4}
                       &  \multirow{2}{*}{$n$ odd}&$s$ odd&$=\lceil \frac{\lambda}{2} \rceil + \frac{s-1}{2} -( \min(\lceil \frac{\gamma+1}{2}\rceil,\lceil \frac{s-\gamma+2}{2}\rceil) -1)$\\
    \cline{3-4}
                     &    &$s$ even&$=\begin{cases} 
                      \lceil \frac{\lambda}{2} \rceil + \frac{s-2}{2} & \mbox{if   $\gamma =1$,} \\
      \lfloor \frac{\lambda}{2} \rfloor + \frac{s-\gamma+1}{2}  & \mbox{if   $3\leq \gamma \leq 2\lceil \frac{s}{4} \rceil -1$,} \\ 
     \lceil \frac{\lambda}{2} \rceil + \frac{\gamma-1}{2}  &  \mbox{otherwise.} 
 \end{cases}$\\
  \hline

  \multicolumn{2}{|c}{$\centering$ $\centering$ $\centering$ $\centering$ $\centering$ $\centering$ $\centering$ $\centering$ $\centering$ $\centering$ $\centering$ $\centering$ $\centering$ $\centering$ $\centering$ $\centering$ $\centering$ $\centering$  $\lambda\leq\gamma$ and $b\leq a\lambda+1$}  & & $=\begin{cases} 
p_1 -1  & \mbox{if  $p_1=p_2$  and  $(\gamma+b)(a\lambda-\lambda+1) \equiv 1 \pmod 2$,} \\
e_1  & \mbox{otherwise.} \end{cases}$ \\
   \hline

  \multicolumn{2}{|c}{$\centering$ $\centering$ $\centering$ $\centering$ $\centering$ $\centering$ $\centering$ $\centering$ $\centering$ $\centering$ $\centering$ $\centering$ $\centering$ $\centering$ $\centering$ $\centering$ $\centering$ $\centering$  All $n$ and $s$}  & & $\leq \min( \max\{ \lfloor \frac{n}{s} \rfloor+1, n-\lfloor \frac{n}{s} \rfloor s-2, ( \lfloor \frac{n}{s} \rfloor+1)s-n-1\}, \lfloor \frac{n+2}{4}\rfloor, \lfloor \frac{\lfloor \frac{n}{2} \rfloor}{s} \rfloor + \lceil \frac{s}{2} \rceil)$ \\
  
  \hline 
  \end{tabular}  
\label{tab1}
\end{table}

\section{Algorithm computing $diam(C_n(1,s))$}
\label{Algo}

Several different approaches have been used to obtain diameter results for circulant graphs. Of the methods we wish to mention, those achieved by Wong and Coppersmith \cite{wong1974combinatorial} were obtained first. Although some of their techniques were later improved upon, their method is quite accessible. Before revealing our approach, let us introduce some definitions and notations related to circulant graphs.

\begin{definition}

We denote a path leading from a vertex $i$ to another vertex $j$ in $C_n(1,s)$ by $P(i,j)$. It is represented by a couple $(\alpha a^{\pm}, \beta c^{\pm})$ where

\begin{itemize}
\item $a$ (resp. $c$) indicates that $P(i,j)$ walks through outer (resp. inner) edges;

\item $\alpha$ (resp. $\beta$) is the number of outer (resp. inner) edges;

\item $+$ (resp. $-$) means that $P(i,j)$ takes the clockwise (resp. the counterclockwise) direction.
\end{itemize}
\end{definition}

\begin{remark} Let $i,j$ be two vertices of $C_n(1,s)$. A path $P(i,j)$ can join the vertices $i$ and $j$ in several ways:
\begin{itemize}
\item $P(i,j)=(\alpha a^{\pm}, \beta c^{\pm})$ means that $P(i,j)$ walks through $\alpha$ outer edges in the clockwise $(+)$ or the counterclockwise $(-)$ direction \textbf{before} going through  $\beta$ inner edges in the clockwise $(+)$ or the counterclockwise $(-)$ direction;

\item $P(i,j)=(\beta c^{\pm}, \alpha a^{\pm})$ means that $P(i,j)$ walks through $\beta$ inner edges in the clockwise $(+)$ or the counterclockwise $(-)$ direction \textbf{before} going $\alpha$ outer edges in the clockwise $(+)$ or the counterclockwise $(-)$ direction;

\item $P(i,j)=(\alpha a^{\pm}, 0)$ means that $P(i,j)$ walks \textbf{only} through $\alpha$ outer edges in the clockwise $(+)$ or the counterclockwise $(-)$ direction;

\item $P(i,j)=(0, \beta c^{\pm})$ means that $P(i,j)$ walks \textbf{only} through $\beta$ inner edges in the clockwise $(+)$ or the counterclockwise $(-)$ direction.
\end{itemize}
\end{remark}

\begin{notation}

In $C_n(1, s)$, we represent an outer (resp. inner) edge connecting the vertices $i$ and $j$ and taking the clockwise $(+)$ or the counterclockwise $(-)$ direction by $i \leadsto^{a^\pm} j$ (resp. $i \leadsto^{c^\pm} j$).
\end{notation}

\begin{notation}

Let $i$ and $j$ be two vertices of $C_n(1, s)$. We denote the length of $P(i,j)$ by $\ell(P(i,j)),$ the number of outer edges of $P(i,j)$ by $\ell_a(P(i,j))$ and the number of inner edges of $P(i,j)$ by $\ell_c(P(i,j)).$
\end{notation}

\begin{example}

Let us focus on the graph $C_{10}(1,4)$ presented in Figure \ref{fig1}. 
\begin{figure}[!h]
\centering
\begin{tikzpicture}[line cap=round,line join=round,>=triangle 45,x=0.86cm,y=0.86cm]
\begin{scriptsize}
\draw [line width=0.8pt] (1.860739087062379,1.351906080272688)-- (-2.3,0.);
\draw [line width=0.8pt] (0.7107390870623791,2.1874299874788528)-- (-1.860739087062379,-1.3519060802726879);
\draw [line width=0.8pt] (-0.7107390870623789,2.187429987478853)-- (-0.7107390870623793,-2.1874299874788528);
\draw [line width=0.8pt] (-1.8607390870623788,1.3519060802726883)-- (0.7107390870623785,-2.187429987478853);
\draw [line width=0.8pt] (-2.3,0.)-- (1.8607390870623788,-1.3519060802726885);
\draw [line width=0.8pt] (-1.860739087062379,-1.3519060802726879)-- (2.3,0.);
\draw [line width=0.8pt] (-0.7107390870623793,-2.1874299874788528)-- (1.860739087062379,1.351906080272688);
\draw [line width=0.8pt] (0.7107390870623785,-2.187429987478853)-- (0.7107390870623791,2.1874299874788528);
\draw [line width=0.8pt] (1.8607390870623788,-1.3519060802726885)-- (-0.7107390870623789,2.187429987478853);
\draw [line width=0.8pt] (2.3,0.)-- (-1.8607390870623788,1.3519060802726883);
\draw [line width=0.8pt] (0.7107390870623791,2.1874299874788528)-- (-1.860739087062379,-1.3519060802726879);
\draw [line width=0.8pt] (-0.7107390870623789,2.187429987478853)-- (-0.7107390870623793,-2.1874299874788528);
\draw [line width=0.8pt] (-1.8607390870623788,1.3519060802726883)-- (0.7107390870623785,-2.187429987478853);
\draw [line width=0.8pt] (-2.3,0.)-- (1.8607390870623788,-1.3519060802726885);
\draw [line width=0.8pt] (-1.860739087062379,-1.3519060802726879)-- (2.3,0.);
\draw [line width=0.8pt] (-0.7107390870623793,-2.1874299874788528)-- (1.860739087062379,1.351906080272688);
\draw [line width=0.8pt] (0.7107390870623785,-2.187429987478853)-- (0.7107390870623791,2.1874299874788528);
\draw [line width=0.8pt] (1.8607390870623788,-1.3519060802726885)-- (-0.7107390870623789,2.187429987478853);
\draw [line width=0.8pt] (2.3,0.)-- (-1.8607390870623788,1.3519060802726883);
\draw [line width=0.8pt] (-0.7107390870623789,2.187429987478853)-- (-0.7107390870623793,-2.1874299874788528);
\draw [line width=0.8pt] (-1.8607390870623788,1.3519060802726883)-- (0.7107390870623785,-2.187429987478853);
\draw [line width=0.8pt] (-2.3,0.)-- (1.8607390870623788,-1.3519060802726885);
\draw [line width=0.8pt] (-1.860739087062379,-1.3519060802726879)-- (2.3,0.);
\draw [line width=0.8pt] (-0.7107390870623793,-2.1874299874788528)-- (1.860739087062379,1.351906080272688);
\draw [line width=0.8pt] (0.7107390870623785,-2.187429987478853)-- (0.7107390870623791,2.1874299874788528);
\draw [line width=0.8pt] (1.8607390870623788,-1.3519060802726885)-- (-0.7107390870623789,2.187429987478853);
\draw [line width=0.8pt] (2.3,0.)-- (-1.8607390870623788,1.3519060802726883);
\draw [line width=0.8pt] (-1.8607390870623788,1.3519060802726883)-- (0.7107390870623785,-2.187429987478853);
\draw [line width=0.8pt] (-2.3,0.)-- (1.8607390870623788,-1.3519060802726885);
\draw [line width=0.8pt] (-1.860739087062379,-1.3519060802726879)-- (2.3,0.);
\draw [line width=0.8pt] (-0.7107390870623793,-2.1874299874788528)-- (1.860739087062379,1.351906080272688);
\draw [line width=0.8pt] (0.7107390870623785,-2.187429987478853)-- (0.7107390870623791,2.1874299874788528);
\draw [line width=0.8pt] (1.8607390870623788,-1.3519060802726885)-- (-0.7107390870623789,2.187429987478853);
\draw [line width=0.8pt] (2.3,0.)-- (-1.8607390870623788,1.3519060802726883);
\draw [line width=0.8pt] (-2.3,0.)-- (1.8607390870623788,-1.3519060802726885);
\draw [line width=0.8pt] (-1.860739087062379,-1.3519060802726879)-- (2.3,0.);
\draw [line width=0.8pt] (-0.7107390870623793,-2.1874299874788528)-- (1.860739087062379,1.351906080272688);
\draw [line width=0.8pt] (0.7107390870623785,-2.187429987478853)-- (0.7107390870623791,2.1874299874788528);
\draw [line width=0.8pt] (1.8607390870623788,-1.3519060802726885)-- (-0.7107390870623789,2.187429987478853);
\draw [line width=0.8pt] (2.3,0.)-- (-1.8607390870623788,1.3519060802726883);
\draw [line width=0.8pt] (-1.860739087062379,-1.3519060802726879)-- (2.3,0.);
\draw [line width=0.8pt] (-0.7107390870623793,-2.1874299874788528)-- (1.860739087062379,1.351906080272688);
\draw [line width=0.8pt] (0.7107390870623785,-2.187429987478853)-- (0.7107390870623791,2.1874299874788528);
\draw [line width=0.8pt] (1.8607390870623788,-1.3519060802726885)-- (-0.7107390870623789,2.187429987478853);
\draw [line width=0.8pt] (2.3,0.)-- (-1.8607390870623788,1.3519060802726883);
\draw [line width=0.8pt] (-0.7107390870623793,-2.1874299874788528)-- (1.860739087062379,1.351906080272688);
\draw [line width=0.8pt] (0.7107390870623785,-2.187429987478853)-- (0.7107390870623791,2.1874299874788528);
\draw [line width=0.8pt] (1.8607390870623788,-1.3519060802726885)-- (-0.7107390870623789,2.187429987478853);
\draw [line width=0.8pt] (2.3,0.)-- (-1.8607390870623788,1.3519060802726883);
\draw [line width=0.8pt] (0.7107390870623785,-2.187429987478853)-- (0.7107390870623791,2.1874299874788528);
\draw [line width=0.8pt] (1.8607390870623788,-1.3519060802726885)-- (-0.7107390870623789,2.187429987478853);
\draw [line width=0.8pt] (2.3,0.)-- (-1.8607390870623788,1.3519060802726883);
\draw [line width=0.8pt] (1.8607390870623788,-1.3519060802726885)-- (-0.7107390870623789,2.187429987478853);
\draw [line width=0.8pt] (2.3,0.)-- (-1.8607390870623788,1.3519060802726883);
\draw [line width=0.8pt] (2.3,0.)-- (-1.8607390870623788,1.3519060802726883);
\draw [line width=1.2pt] (1.860739087062379,1.351906080272688)-- (0.7107390870623791,2.1874299874788528);
\draw [line width=1.2pt] (0.7107390870623791,2.1874299874788528)-- (-0.7107390870623789,2.187429987478853);
\draw [line width=1.2pt] (-0.7107390870623789,2.187429987478853)-- (-1.8607390870623788,1.3519060802726883);
\draw [line width=1.2pt] (-1.8607390870623788,1.3519060802726883)-- (-2.3,0.);
\draw [line width=1.2pt] (-2.3,0.)-- (-1.860739087062379,-1.3519060802726879);
\draw [line width=1.2pt] (-1.860739087062379,-1.3519060802726879)-- (-0.7107390870623793,-2.1874299874788528);
\draw [line width=1.2pt] (-0.7107390870623793,-2.1874299874788528)-- (0.7107390870623785,-2.187429987478853);
\draw [line width=1.2pt] (0.7107390870623785,-2.187429987478853)-- (1.8607390870623788,-1.3519060802726885);
\draw [line width=1.2pt] (1.8607390870623788,-1.3519060802726885)-- (2.3,0.);
\draw [line width=1.2pt] (2.3,0.)-- (1.860739087062379,1.351906080272688);
\draw [line width=1.2pt] (0.7107390870623791,2.1874299874788528)-- (-0.7107390870623789,2.187429987478853);
\draw [line width=1.2pt] (-0.7107390870623789,2.187429987478853)-- (-1.8607390870623788,1.3519060802726883);
\draw [line width=1.2pt] (-1.8607390870623788,1.3519060802726883)-- (-2.3,0.);
\draw [line width=1.2pt] (-2.3,0.)-- (-1.860739087062379,-1.3519060802726879);
\draw [line width=1.2pt] (-1.860739087062379,-1.3519060802726879)-- (-0.7107390870623793,-2.1874299874788528);
\draw [line width=1.2pt] (-0.7107390870623793,-2.1874299874788528)-- (0.7107390870623785,-2.187429987478853);
\draw [line width=1.2pt] (0.7107390870623785,-2.187429987478853)-- (1.8607390870623788,-1.3519060802726885);
\draw [line width=1.2pt] (1.8607390870623788,-1.3519060802726885)-- (2.3,0.);
\draw [line width=1.2pt] (2.3,0.)-- (1.860739087062379,1.351906080272688);
\draw [line width=1.2pt] (-0.7107390870623789,2.187429987478853)-- (-1.8607390870623788,1.3519060802726883);
\draw [line width=1.2pt] (-1.8607390870623788,1.3519060802726883)-- (-2.3,0.);
\draw [line width=1.2pt] (-2.3,0.)-- (-1.860739087062379,-1.3519060802726879);
\draw [line width=1.2pt] (-1.860739087062379,-1.3519060802726879)-- (-0.7107390870623793,-2.1874299874788528);
\draw [line width=1.2pt] (-0.7107390870623793,-2.1874299874788528)-- (0.7107390870623785,-2.187429987478853);
\draw [line width=1.2pt] (0.7107390870623785,-2.187429987478853)-- (1.8607390870623788,-1.3519060802726885);
\draw [line width=1.2pt] (1.8607390870623788,-1.3519060802726885)-- (2.3,0.);
\draw [line width=1.2pt] (2.3,0.)-- (1.860739087062379,1.351906080272688);
\draw [line width=1.2pt] (-1.8607390870623788,1.3519060802726883)-- (-2.3,0.);
\draw [line width=1.2pt] (-2.3,0.)-- (-1.860739087062379,-1.3519060802726879);
\draw [line width=1.2pt] (-1.860739087062379,-1.3519060802726879)-- (-0.7107390870623793,-2.1874299874788528);
\draw [line width=1.2pt] (-0.7107390870623793,-2.1874299874788528)-- (0.7107390870623785,-2.187429987478853);
\draw [line width=1.2pt] (0.7107390870623785,-2.187429987478853)-- (1.8607390870623788,-1.3519060802726885);
\draw [line width=1.2pt] (1.8607390870623788,-1.3519060802726885)-- (2.3,0.);
\draw [line width=1.2pt] (2.3,0.)-- (1.860739087062379,1.351906080272688);
\draw [line width=1.2pt] (-2.3,0.)-- (-1.860739087062379,-1.3519060802726879);
\draw [line width=1.2pt] (-1.860739087062379,-1.3519060802726879)-- (-0.7107390870623793,-2.1874299874788528);
\draw [line width=1.2pt] (-0.7107390870623793,-2.1874299874788528)-- (0.7107390870623785,-2.187429987478853);
\draw [line width=1.2pt] (0.7107390870623785,-2.187429987478853)-- (1.8607390870623788,-1.3519060802726885);
\draw [line width=1.2pt] (1.8607390870623788,-1.3519060802726885)-- (2.3,0.);
\draw [line width=1.2pt] (2.3,0.)-- (1.860739087062379,1.351906080272688);
\draw [line width=1.2pt] (-1.860739087062379,-1.3519060802726879)-- (-0.7107390870623793,-2.1874299874788528);
\draw [line width=1.2pt] (-0.7107390870623793,-2.1874299874788528)-- (0.7107390870623785,-2.187429987478853);
\draw [line width=1.2pt] (0.7107390870623785,-2.187429987478853)-- (1.8607390870623788,-1.3519060802726885);
\draw [line width=1.2pt] (1.8607390870623788,-1.3519060802726885)-- (2.3,0.);
\draw [line width=1.2pt] (2.3,0.)-- (1.860739087062379,1.351906080272688);
\draw [line width=1.2pt] (-0.7107390870623793,-2.1874299874788528)-- (0.7107390870623785,-2.187429987478853);
\draw [line width=1.2pt] (0.7107390870623785,-2.187429987478853)-- (1.8607390870623788,-1.3519060802726885);
\draw [line width=1.2pt] (1.8607390870623788,-1.3519060802726885)-- (2.3,0.);
\draw [line width=1.2pt] (2.3,0.)-- (1.860739087062379,1.351906080272688);
\draw [line width=1.2pt] (0.7107390870623785,-2.187429987478853)-- (1.8607390870623788,-1.3519060802726885);
\draw [line width=1.2pt] (1.8607390870623788,-1.3519060802726885)-- (2.3,0.);
\draw [line width=1.2pt] (2.3,0.)-- (1.860739087062379,1.351906080272688);
\draw [line width=1.2pt] (1.8607390870623788,-1.3519060802726885)-- (2.3,0.);
\draw [line width=1.2pt] (2.3,0.)-- (1.860739087062379,1.351906080272688);
\draw [line width=1.2pt] (2.3,0.)-- (1.860739087062379,1.351906080272688);
\draw (0.6739544712378257,2.5972681813603957) node[anchor=north west] {\textbf{0}};
\draw (-1,2.6) node[anchor=north west] {\textbf{9}};
\draw (1.9,1.6) node[anchor=north west] {\textbf{1}};
\draw (2.4,0.1385697063923914) node[anchor=north west] {\textbf{2}};
\draw (2,-1.3477621400571933) node[anchor=north west] {\textbf{3}};
\draw (0.5,-2.43) node[anchor=north west] {\textbf{4}};
\draw (-0.8957230862276198,-2.417365431427455) node[anchor=north west] {\textbf{5}};
\draw (-2.4,-1.2644164290413287) node[anchor=north west] {\textbf{6}};
\draw (-2.7293287285766357,0.2) node[anchor=north west] {\textbf{7}};
\draw (-2.35,1.6387925046779535) node[anchor=north west] {\textbf{8}};
\draw [fill=black] (1.860739087062379,1.351906080272688) circle (2.0pt);
\draw [fill=black] (0.7107390870623791,2.1874299874788528) circle (2.0pt);
\draw [fill=black] (-0.7107390870623789,2.187429987478853) circle (2.0pt);
\draw [fill=black] (-1.8607390870623788,1.3519060802726883) circle (2.0pt);
\draw [fill=black] (-2.3,0.) circle (2.0pt);
\draw [fill=black] (-1.860739087062379,-1.3519060802726879) circle (2.0pt);
\draw [fill=black] (-0.7107390870623793,-2.1874299874788528) circle (2.0pt);
\draw [fill=black] (0.7107390870623785,-2.187429987478853) circle (2.0pt);
\draw [fill=black] (1.8607390870623788,-1.3519060802726885) circle (2.0pt);
\draw [fill=black] (2.3,0.) circle (2.0pt);
\end{scriptsize}
\end{tikzpicture}
\caption{The circulant graph $C_{10}(1,4).$}
\label{fig1}
\end{figure}
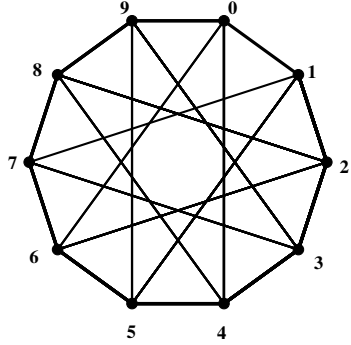

If we chose an arbitrary vertex, for instance the vertex $6.$ there exist various paths leading from $0$ to $6.$ We give some of them as an example in Table \ref{tab2},

\begin{table}[!h]
\centering
\begin{tabular}{|l|c|c|c|c|}

\hline

$P(0,6)$ & Representation of $P(0,6)$  & $\ell_a(P(0,6))$ & $\ell_c(P(0,6))$ & $\ell(P(0,6))$\\

\hline

$=(2a^+,1c^+)$ & $0 \leadsto^{a^+} 1 \leadsto^{a^+} 2 \leadsto^{c^+} 6$  & 2 & 1 & 3\\

\hline

$=(2a^-,2c^+)$ & $0 \leadsto^{a^-} 9 \leadsto^{a^-} 8 \leadsto^{c^+} 2 \leadsto^{c^+} 6$  & 2 & 2 & 4\\

\hline

$=(0,4c^+)$ & $0 \leadsto^{c^+} 4 \leadsto^{c^+} 8 \leadsto^{c^+} 2 \leadsto^{c^+} 6$  & 0 & 4 & 4\\

\hline

$=(0,c^-)$ & $0 \leadsto^{c^-} 6$ & 0 & 1 & 1 \\

\hline
\end{tabular}
\caption{Examples of paths leading from the vertex $0$ to the vertex $6$ in $C_{10}(1,4).$}
\label{tab2}
\end{table}
\end{example}


\begin{definition}
Let $i,j \in V(C_n(1,s)).$ Let $P(i,j)$ and $Q(i,j)$ be two paths in $C_n(1,s)$. $P(i,j)$ is equivalent to $Q(i,j)$, denoted $P(i,j) \approx Q(i,j),$ if and only if $\ell(P(i,j))=\ell(Q(i,j)),$ $\ell_a(P(i,j))=\ell_a(Q(i,j)),$ $\ell_c(P(i,j))=\ell_c(Q(i,j)),$ and that $P(i,j)$ and $Q(i,j)$ take the same direction.
\end{definition}

\begin{lemma}
Let $i,j \in V(C_n(1,s)).$ Let $P(i,j)$ and $Q(i,j)$ be two paths in $C_n(1,s)$. $\approx$ is an equivalence relation on the set of paths in $C_n(1,s)$.
\end{lemma}

\begin{proof}
Let $\mathcal{P}$ be the set of paths in $C_n(1,s)$. It is easy to verify that $\approx$ is an equivalence relation on $\mathcal{P}.$
\end{proof}

\begin{lemma}
Let $i$ and $j$ be two vertices of $C_n(1,s).$ We have
$$
P(i,j)\approx\begin{cases}
			P(0,j-i), & \text{if $ i< j$}\\
            P(0,n-i+j), & \text{otherwise}
		 \end{cases}
$$
\end{lemma}

\begin{proof}
The circulant graph $C_n(1,s)$ is vertex-transitive. Thus, for any two vertices $i$ and $j$ of $C_n(1,s),$ the path $P(i,j)$ can be translated into the path $P(0,k)$ where 
$$
k=\begin{cases}
			j-i, & \text{if $ i< j$}\\
            n-i+j, & \text{otherwise}
		 \end{cases}
$$
\end{proof}

\begin{remark}
For the rest of this work,  we consider paths leading from $0$ to a vertex $i$ of $C_n(1,s)$, denoted $P(i).$ We denote the distance of $P(i)$ by $d(i)$.
\end{remark}

\begin{example}
Let us take the graph  $C_{10}(1,4)$ presented in Figure \ref{fig1}. Here is an example of some equivalent paths in $C_{10}(1,4).$
\begin{itemize}
\item $P(6,9)=(1a^-,1c^+)$ is equivalent to $P(3)=(1a^-,1c^+);$

\item $P(6,2)=(0,1c^-)$ is equivalent to $P(6)=(0,1c^-);$
\end{itemize}
\end{example}

\begin{lemma} \label{Key}
For any vertex $i$ of $C_n(1,s)$, there exists a path leading from $0$ to $i$ by walking through all its outer edges \textbf{before} entering to its inner edges or vice versa.
\end{lemma}

\begin{proof}
Let $i$ be a vertex of $C_n(1,s)$. Without loss of generality, if we assume that $W$ is an arbitrary walk  in $C_n(1,s)$ leading from $0$ to $i$ by taking $ \alpha a^{+}, $ $ \beta a^{-},$ $ \gamma c^{+}$  and $ \lambda c^{-}. $ Then, it exists a path $P(i)$, walking through all its outer edges \textit{\textbf{before}} entering to its inner edges or vice versa, represented as follows.

\begin{align*}
P(i) &= 
 \begin{cases} 
      ((\alpha-\beta) a^+ , (\gamma -\lambda) c^+) & \qquad \mbox{if } \quad \alpha\geq\beta \ \mbox{ and  } \ \gamma \geq\lambda, \\ 
      ((\beta-\alpha) a^- , (\lambda-\gamma) c^-) & \qquad \mbox{if } \quad \alpha\leq\beta \ \mbox{ and  } \ \gamma \leq\lambda, \\
      ((\alpha-\beta) a^+ , (\lambda-\gamma) c^-) & \qquad \mbox{if } \quad \alpha\geq\beta \ \mbox{ and } \ \gamma \leq\lambda, \\ 
      ((\beta-\alpha) a^- , (\gamma-\lambda) c^+) & \qquad \mbox{if } \quad \alpha\leq\beta \ \mbox{ and } \ \gamma \geq\lambda. 
 \end{cases}
\end{align*}\\ Thus, $\ell(W)=\alpha + \beta + \gamma + \lambda$ and, \begin{align*}
\ell(P(i)) &= 
 \begin{cases} 
      (\alpha-\beta) + (\gamma -\lambda) & \qquad \mbox{if } \quad \alpha\geq\beta \ \mbox{ and  } \ \gamma \geq\lambda, \\ 
      (\beta-\alpha) + (\lambda-\gamma)  & \qquad \mbox{if } \quad \alpha\leq\beta \ \mbox{ and  } \ \gamma \leq\lambda, \\
      (\alpha-\beta) + (\lambda-\gamma)  & \qquad \mbox{if } \quad \alpha\geq\beta \ \mbox{ and } \ \gamma \leq\lambda, \\ 
      (\beta-\alpha) + (\gamma-\lambda)  & \qquad \mbox{if } \quad \alpha\leq\beta \ \mbox{ and } \ \gamma \geq\lambda. 
 \end{cases}
\end{align*}

It is obvious that $\ell(P(i))\leq \ell(W).$ Thus, for all $i\in V(C_n(1,s))$, there exists a path leading from $0$ to $i$ by walking through all its outer edges \textbf{before} entering to its inner edges or vice versa.
\end{proof}

\begin{example}
Let us take the graph  $C_{10}(1,4)$ presented in Figure \ref{fig1}. Without loss of generality, if we assume that $W$ is an arbitrary walk  in $C_n(1,s)$ leading from $0$ to $6$ by taking $1 a^{+}, $ $ 2 c^{+}$,  $ 1 a^{-},$   and $ 3 c^{-}. $ Then,  $W$ is represented as follows. $$0 \leadsto^{a^+} 1 \leadsto^{c^+} 5 \leadsto^{c^+} 9 \leadsto^{a^-} 8 \leadsto^{c^-}  4 \leadsto^{c^-} 0 \leadsto^{c^-} 6.$$  By Lemma \ref{Key}, there exists a path leading from $0$ to $6$ defined by $P(6)=((1-1)a^+ ,(3-2)c^-)=(0,1c^-).$ We have $\ell(P(6)=1<\ell(w)=7$. Thus, for the vertex $6$, there exists a path $P(6)$ walking through all its outer edges \textit{\textbf{before}} entering to its inner edges. 
\end{example}

\begin{notation}
Let $t$ be an integer such that $1\leq t\leq \frac{s}{gcd(n,s)}.$  we denote the integers $q,q_t,\bar{q_t},r,r_t,$ and $\bar{r_t}$ by
\begin{itemize}
\item $i\equiv r \pmod{s}$ and $q:=\lfloor \frac{i}{s}\rfloor$,

\item $tn+i\equiv r_t \pmod{s}$ and $q_t:=\lfloor \frac{tn+i}{s}\rfloor$,

\item $tn-i\equiv \bar{r}_t \pmod{s}$ and $\bar{q}_t:=\lfloor \frac{tn-i}{s}\rfloor$.
\end{itemize}
\end{notation}

\begin{remark}
In $C_n(1,s)$, the inner edges form $gcd(n, s)$ cycles of length $\frac{n}{gcd(n,s)}.$
\end{remark}

\begin{remark}
Let $i$ be a vertex of $C_n(1,s)$.
\begin{itemize}
\item If $gcd(n,s)=1,$ then the number of inner edges in $P(i)$ cannot be greater than $n-1$, i.e., $1\leq \bar{q}_{t} \leq n-1$ and $ 1 \leq q_{t} \leq n-1.$ Therefore, $1\leq t \leq s.$

\item Otherwise,  if $gcd(n,s)\neq 1$ then  $1\leq \bar{q}_{t}  < \frac{n}{gcd(n,s)}$ and $1\leq q_{t}  < \frac{n}{gcd(n,s)}.$  Therefore, $1\leq t \leq \frac{s}{gcd(n,s)}.$
\end{itemize}
\end{remark}

Our goal lies in providing formulas giving exact values for the diameter of $C_n(1,s)$ for all values of $n$ and $s$. Since by definition, the diameter is the maximum distance among all pair of vertices  and, the distance is the shortest path, we analyzed the behavior of paths in $C_n(1,s)$, and we found out that there is plenty of paths leading from $0$ to a vertex $i$ in $C_n(1,s)$ that are equivalent. Our approach lies in determining this equivalent class of paths in $C_n(1,s)$.

\begin{lemma}
\label{path}
Let $t$ be an integer such that $1\leq t\leq \frac{s}{gcd(n,s)}.$ For every vertex $i$  of $C_n(1,s)$, there exists a class of pairwise non-equivalent paths  $\mathcal{C}= \{P^1 (i) , P^2 (i), P^{1,t}(i),$ $ P^{2,t}(i), P^{3,t}(i), P^{4,t}(i) \}$ such that

\vskip 0.1cm  $P^1 (i)=(ra^+,qc^+) \qquad \qquad \quad \qquad \ \  \ P^2 (i)=((s-r)a^-,(q+1)c^+) \qquad \ \  \  P^{1,t}(i)=(r_{t}a^+,q_{t}c^+);$

\vskip 0.1cm  $
  P^{2,t}(i)=((s-r_{t})a^-,(q_{t}+1)c^+) \qquad \ \ P^{3,t}(i)=(\bar{r}_{t}a^-,\bar{q}_{t}c^-) \qquad \qquad \qquad P^{4,t}(i)=((s-\bar{r}_{t})a^+,(\bar{q}_{t}+1)c^-). $ \\ The length of these paths is consecutively equals to
\vskip 0.1cm   $ \ell^1 (i)=r+ q \quad \quad  \qquad \qquad \quad \ \ \ell^2 (i)=1+s-r+ q \quad \quad \quad \ell^{1,t} (i)=r_{t}+ q_{t}   ;$

\vskip 0.1cm  $ \ell^{2,t} (i)=1+s-r_{t}+ q_{t} \quad \quad \quad \ell^{3,t} (i)=\bar{r}_{t}+ \bar{q}_{t} \quad \quad \qquad \ \ \ell^{4,t} (i)=1+s-\bar{r}_{t}+ \bar{q}_{t}$.
\end{lemma}

\begin{proof}
Let $i$ be a vertex of $C_n(1,s)$ and $t$ be an integer such that $1\leq t\leq \frac{s}{gcd(n,s)}.$

As $i=q s+r,$ we consider the path $P(i)$ represented as follows 
   	$0 \leadsto^{a^+} 1 \leadsto^{a^+} 2 \leadsto^{a^+} \dots \leadsto^{a^+} r \leadsto^{c^+} r+s \leadsto^{c^+} r+2s \leadsto^{c^+} \dots \leadsto^{c^+} r+q s=i.$ We have $P^1 (i)=(r a^+,q c^+)$  	and, $\ell^1 (i)= r+ q$.

$i$ can  be written as $i=(q+1) s+(r-s).$ We consider, in this case, the path $P(i)$ represented as follows 
   	$0 \leadsto^{a^-} n-1 \leadsto^{a^-} n-2 \leadsto^{a^-} \dots \leadsto^{a^-} n-(s-r) \leadsto^{c^+} n-(s-r)+s \leadsto^{c^+} n-(s-r)+2s \leadsto^{c^+} \dots \leadsto^{c^+} n-(s-r)+(q+1) s=n+i  \ (\equiv i \mod n).$ We have $P^2(i)=((s-r)a^-,(q+1)c^+)$  	and,  $\ell^2 (i)=1+s-r+ q$.

$i$ can also be written as $tn+i=q_{t} s+r_{t}$. In this case, we consider the path $P(i)$ represented as follows  	$0 \leadsto^{a^+} 1 \leadsto^{a^+} 2 \leadsto^{a^+} \dots \leadsto^{a^+} r_{t} \leadsto^{c^+} r_{t}+s \leadsto^{c^+} r_{t}+2s \leadsto^{c^+} \dots \leadsto^{c^+} r_{t}+q_{t} s=tn+i \ (\equiv i \mod n).$   We have $P^{3,t}(i)=(r_{t}a^+,q_{t}c^+)$ and, $\ell^{1,t} (i)=r_{t}+ q_{t}.$

$i$ can be written as $tn+i=(q_{t}+1) s+(r_{t}-s)$. We consider the path $P(i)$ represented as follows  	  
  	$0 \leadsto^{a^-} n-1 \leadsto^{a^-} n-2 \leadsto^{a^-} \dots \leadsto^{a^-} n-(s-r_{t}) \leadsto^{c^+} n-(s-r_{t})+s \leadsto^{c^+} n-(s-r_{t})+2s \leadsto^{c^+} \dots \leadsto^{c^+} n-(s-r_{t})+(q_{t}+1) s=(t+1)n+i  \ (\equiv i \mod n).$ We  have $P^{2,t}(i)=((s-r_{t})a^-,(q_{t}+1)c^+)$	and, $\ell^{2,t} (i)=1+s-r_{t}+ q_{t}.$

$i$ can be written as $tn-i=\bar{q}_{t} s+\bar{r}_{t}$. We consider the path $P(i)$ represented as follows  
   	  $0 \leadsto^{a^-} n-1 \leadsto^{a^-} n-2 \leadsto^{a^-} \dots \leadsto^{a^-} -\bar{r}_{t} \leadsto^{c^-}  -\bar{r}_{t}-s \leadsto^{c^-} -\bar{r}_{t}-2s \leadsto^{c^-} \dots \leadsto^{c^-} -\bar{r}_{t}-\bar{q}_{t} s=-tn+i  \ (\equiv i \mod n).$  We have $P^{3,t}(i)=(\bar{r}_{t}a^-,\bar{q}_{t}c^-)$ 	and, $\ell^{3,t} (i)=\bar{r}_{t}+ \bar{q}_{t}.$

$i$ can be written as  $tn-i=(\bar{q}_{t}+1) s+(\bar{r}_{t}-s)$. We consider the path $P(i)$ represented as follows  		
   	$0 \leadsto^{a^+} 1 \leadsto^{a^+} 2 \leadsto^{a^+} \dots \leadsto^{a^+} -(\bar{r}_{t}-s) \leadsto^{c^-} -(\bar{r}_{t}-s)-s \leadsto^{c^-} -(\bar{r}_{t}-s)-2s \leadsto^{c^-} \dots \leadsto^{c^-} -(\bar{r}_{t}-s)-(\bar{q}_{t}+1) s=-tn+i \ (\equiv i \mod n). $  We have $P^{4,t}(i)=((s-\bar{r}_{t})a^+,(\bar{q}_{t}+1)c^-)$  	and,  $\ell^{4,t} (i)=1+s-\bar{r}_{t}+ \bar{q}_{t}.$
\end{proof}

\begin{remark}
For some values of $n$ and $s$, it is possible that $\mathcal{C}$ contains some \textit{walks}. If this is the case, then these walks will be equivalent to other paths of $\mathcal{C}$ of smaller lengths. 

\end{remark}

\begin{proposition}

Every single path in the class $\mathcal{C}$ constitute an equivalence class.

\end{proposition}

\begin{proof}
Let $i$ be a vertex of $C_n(1,s)$ and $t$ be an integer such that $1\leq t\leq \frac{s}{gcd(n,s)}.$  The set of paths $\mathcal{P}$ contains the following six equivalence classes: 
\begin{itemize}
\item $[P^1 (i)]=\{P(i) \in \mathcal{P} : P^1 (i)\approx P(i) \}$, 
 $[P^2 (i)]=\{P(i) \in \mathcal{P} : P^2 (i)\approx P(i) \}$;
 \item $[P^{1,t}(i)]=\{P(i) \in \mathcal{P} : P^{1,t}(i)\approx P(i) \}$,
 $[P^{2,t}(i)]=\{P(i) \in \mathcal{P} : P^{2,t}(i)\approx P(i) \}$;
 \item $[P^{3,t}(i)]=\{P(i) \in \mathcal{P} : P^{3,t}(i)\approx P(i) \}$,  $[P^{4,t}(i)]=\{P(i) \in \mathcal{P} : P^{4,t}(i)\approx P(i) \}$.
\end{itemize}
  In fact and without loss of generality, if we assume that $W$ is an arbitrary walk  in $C_n(1,s)$ leading from $0$ to $i$, then, by Lemma \ref{Key}, there exists a  path $Q(i)$ walking through all its outer edges \textbf{before} entering to its inner edges or vice versa.  Without loss of generality, assume that $Q(i)=(\alpha a^\pm , \beta c^\pm)$. In this case, there exists  $ P(i) \in \mathcal{C}$ such that $\ell(P(i))=\ell(Q(i)),$ $\ell_a(P(i))=\ell_a(Q(i)),$ $ \ell_c(P(i))=\ell_c(Q(i))$, and $P(i)$ and $Q(i)$ take the same direction,  i.e, there exist  $ P(i) \in \mathcal{C}$ such that $ P(i) \approx Q(i)$.
\end{proof}

\begin{example}
Let us take the graph  $C_{10}(1,4)$ presented in Figure \ref{fig1}. Without loss of generality, if we assume that $W$ is an arbitrary walk  in $C_n(1,s)$ leading from $0$ to $6$ by taking $1 a^{+}, $ $ 2 c^{+}$,  $ 1 a^{-},$   and $ 3 c^{-}. $ The walk  $W$ is represented as follows. $$0 \leadsto^{a^+} 1 \leadsto^{c^+} 5 \leadsto^{c^+} 9 \leadsto^{a^-} 8 \leadsto^{c^-}  4 \leadsto^{c^-} 0 \leadsto^{c^-} 6.$$ By Lemma \ref{Key}, there exists a path $Q(6)$ such that $Q(6)=((1-1)a^+ ,(3-2)c^-)=(0,1c^-).$ We remark that $\ell(Q(6))=1<\ell(W)=7.$ Furthermore, by Lemma \ref{path}, for $t=1$, we have $n-i=10-6=4=s$. Thus, $P^{3,1}(6)=(0,1c^-).$ Hence, $ Q(6) \approx P^{3,1}(6)$. Consequently, $ Q(6) \in [P^{3,1}(6)]$.
\end{example}

Since, by definition, the distance is the shortest path, the next result provides the distance between $0$ and any vertex $i$ in $C_n(1,s)$.

\begin{lemma}
\label{dis1}
Let $t$ be an integer such that $1\leq t\leq \frac{s}{gcd(n,s)}.$ For all $ i\in V(C_n(1,s))$, we have $$d(i) =\min(\ell^1 (i), \ell^2 (i), \ell^{1,t} (i),  \ell^{2,t} (i),  \ell^{3,t} (i),  \ell^{4,t} (i)).$$ 

\end{lemma}

As $diam(C_n(1,s))= \max\{d(i):  \ 0\leq i \leq n\}$, we obtain the following result.

\begin{lemma}\label{dis2}
For all $n$ and $s$, we have
$$diam(C_n(1,s))= \max\{d(i):  \ 2\leq i \leq \lfloor \frac{n}{2} \rfloor\}.$$
\end{lemma}

\begin{proof}
Let $t$ be an integer such that $1\leq t\leq \frac{s}{gcd(n,s)},$ and let $i$ be a vertex of $C_n(1,s)$. We need to prove that  $d(i)=d(n-i)$ for  $1\leq i \leq \lfloor \frac{n}{2} \rfloor$. Since by Lemma \ref{dis1}, $d(i) =\min(\ell^1 (i), \ell^2 (i), \ell^{1,t} (i),  \ell^{2,t} (i),  \ell^{3,t} (i),  \ell^{4,t} (i)),$ we have $d(i) \leq \ell(i)$ where $\ell(i) \in \{ \ell^1 (i), \ell^2 (i), \ell^{1,t} (i),  \ell^{2,t} (i),  \ell^{3,t} (i),  \ell^{4,t} (i) \}$.  First, we prove that  $d(i)\leq d(n-i).$ For that we need to discuss the following cases:

\begin{itemize}
\item If $d(n-i)=\ell^1 (n-i)$, then, we have $\ell^{3,1} (i) = \bar{r}_{1}+ \lfloor \frac{n-i}{s} \rfloor = \ell^1 (n-i)$ and $d(i)\leq \ell^{3,1} (i)$. Thus,  $d(i)\leq d(n-i).$

\item If $d(n-i)=\ell^2 (n-i),$ then, we have $\ell^{4,1} (i) = 1+s-\bar{r}_{1}+ \lfloor \frac{n-i}{s} \rfloor = \ell^2 (n-i)$ and $d(i)\leq \ell^{4,1} (i)$.   Thus, $d(i)\leq d(n-i).$

\item If $d(n-i)=\ell^{1,t} (n-i),$ with $1\leq t \leq \frac{s}{gcd(n,s)}$, then for $t=1,$ we have $\ell^{1,1} (n-i)=r_{1}+ \lfloor \frac{n+(n-i)}{s} \rfloor =  r_{1}+ \lfloor \frac{2n-i}{s} \rfloor=\ell^{3,2} (i)$ and $d(i)\leq \ell^{3,2} (i)$. Thus, $d(i)\leq d(n-i).$  For $t\geq 2,$ we have $\ell^{3,t} (i) = \bar{r}_{t}+ \lfloor \frac{tn-i}{s} \rfloor = \bar{r}_{t}+ \lfloor \frac{(t-1)n+(n-i)}{s} \rfloor =\ell^{1,t-1} (n-i)$ and $d(i)\leq \ell^{3,t} (i)$. Thus, $d(i)\leq d(n-i).$

\item If $d(n-i)=\ell^{2,t} (n-i),$ with $1\leq t \leq \frac{s}{gcd(n,s)}$, then for $t=1,$ $\ell^{2,1} (n-i)=1+s-r_{1}+ \lfloor \frac{n+(n-i)}{s} \rfloor = 1+s- r_{1}+ \lfloor \frac{2n-i}{s} \rfloor=\ell^{4,2} (i)$ and $d(i)\leq \ell^{4,2} (i)$. Thus, $d(i)\leq d(n-i).$  For $t\geq 2,$ $\ell^{4,t} (i) =1+s- \bar{r}_{t}+ \lfloor \frac{tn-i}{s} \rfloor = 1+s-\bar{r}_{t}+ \lfloor \frac{(t-1)n+(n-i)}{s} \rfloor =\ell^{2,t-1} (n-i)$ and $d(i)\leq \ell^{4,t} (i)$. Thus, $d(i)\leq d(n-i).$

\item If $d(n-i)=\ell^{3,t} (n-i),$ with $1\leq t \leq \frac{s}{gcd(n,s)}$, then for $t=1,$ $\ell^{3,1} (n-i)=\bar{r}_{1}+ \lfloor \frac{n-(n-i)}{s} \rfloor =  \bar{r}_{1}+ \lfloor \frac{i}{s} \rfloor=\ell^{1} (i)$ and $d(i)\leq \ell^{1} (i)$. Thus, $d(i)\leq d(n-i).$  For $t\geq 2,$ $\ell^{1,t-1} (i) = r_{t-1}+ \lfloor \frac{(t-1)n+i}{s} \rfloor = r_{t-1}+ \lfloor \frac{tn-(n-i)}{s} \rfloor =\ell^{3,t} (n-i)$ and $d(i)\leq \ell^{1,t-1} (i)$. Thus, $d(i)\leq d(n-i).$

\item If $d(n-i)=\ell^{4,t} (n-i),$ with $1\leq t \leq \frac{s}{gcd(n,s)}$, then for $t=1,$ $\ell^{4,1} (n-i)=1+s-\bar{r}_{1}+ \lfloor \frac{n-(n-i)}{s} \rfloor = 1+s- \bar{r}_{1}+ \lfloor \frac{i}{s} \rfloor=\ell^{2} (i)$ and $d(i)\leq \ell^{2} (i)$. Thus, $d(i)\leq d(n-i).$  For $t\geq 2,$ $\ell^{2,t-1} (i) = 1+s-r_{t-1}+ \lfloor \frac{(t-1)n+i}{s} \rfloor =1+s- r_{t-1}+ \lfloor \frac{tn-(n-i)}{s} \rfloor =\ell^{4,t} (n-i)$ and $d(i)\leq \ell^{2,t-1} (i)$. Thus, $d(i)\leq d(n-i).$ 
\end{itemize}
By the same method, we prove that  $d(i)\geq d(n-i).$ Thus, $diam(C_n(1,s))= \max\{d(i): \ 1\leq i \leq \lfloor \frac{n}{2} \rfloor\}$. However, $d(1)=d(n-1)=1.$ Hence, $diam(C_n(1,s))= \max\{d(i):  \ 2\leq i \leq \lfloor \frac{n}{2} \rfloor\}$.
\end{proof}

Next, we present an algorithm  computing the diameter of $C_n(1,s)$.

\begin{algorithm}[H] \label{MyAlgo}
\SetAlgoLined
\KwResult{Value of $diam(C_n(1,s))$}
\vskip 0.2cm
 \textbf{Step 1.} For given $n$ and $s$, we compute $gcd(n,s)$;
 \vskip 0.2cm
 \textbf{Step 2.} For $i\in \{2, \ldots\lfloor \frac{n}{2} \rfloor\}$ and $1\leq t \leq \frac{s}{gcd(n,s)}$, we compute $$d(i) =\min(\ell^1 (i), \ell^2 (i), \ell^{1,t} (i),  \ell^{2,t} (i),  \ell^{3,t} (i),  \ell^{4,t} (i));$$
 \vskip 0.2cm
 
 \textbf{Step 3.} $diam(C_n(1,s)) = \displaystyle \max_{2\leq i\leq \lfloor \frac{n}{2} \rfloor} d(i)$.
 
 \caption{Diameter of $C_n(1,s)$}
\end{algorithm}

\begin{remark}
Algorithm \ref{MyAlgo} is a simple algorithm that gives exact values for the diameter of $C_n(1,s)$ for all  $n$ and $s.$ 
\end{remark}

 The next section provides formulas giving exact values for the diameter of $C_n(1,s)$, that algorithm \ref{MyAlgo} has established.

\section{Diameter formulas for $diam(C_n(1,s))$}
\label{diam}

 Our method makes it possible to find the vertices of $C_{n}(1,s)$ whose distance with the vertex $0$ coincides with the diameter of the graph, and which allows to have an exact formula for the diameter. Let $n\geq 5$ and $2\leq s\leq \lfloor \frac{n-1}{2} \rfloor.$
For the rest of this section, let $n=\lambda s +\gamma$ where $\gamma$ is an integer such that $0\leq \gamma <s$.
 

\begin{theorem}\cite{chen2005diameter}\label{0}
If $n=\lambda s$,
$$diam(C_{n}(1,s)) = \lfloor \frac{\lambda +s-1}{2}\rfloor.$$
\end{theorem}
 
 Next, we discuss the case when $\gamma\neq 0$ and $\lambda > \gamma$.


\subsection{Diameter of $C_{n}(1,s)$ when $n$ is even, $s$ is odd, and $\lambda > \gamma > 0$}


\begin{lemma}\cite{chen2005diameter}\label{1.1}
Let $n$ be even and $s$ be odd. We have, 
$$diam(C_n(1,s))\leq \lceil \frac{\lambda}{2} \rceil + \frac{s-1}{2} -(\min( \lceil \frac{\gamma}{2} \rceil , \lceil \frac{s-\gamma+1}{2} \rceil) -1).$$
\end{lemma}

\begin{lemma}\label{1.2}
Let $n$ be even and $s$ be odd. There exists a vertex $i$ of $C_{n}(1,s)$ such that, 
$$d(i)= \lceil \frac{\lambda}{2} \rceil + \frac{s-1}{2} -(\min( \lceil \frac{\gamma}{2} \rceil , \lceil \frac{s-\gamma+1}{2} \rceil) -1).$$
\end{lemma}

\begin{proof}
Let $d=\lceil \frac{\lambda}{2} \rceil + \frac{s-1}{2} -(\min( \lceil \frac{\gamma}{2} \rceil , \lceil \frac{s-\gamma+1}{2} \rceil) -1)$. Let $R=\{1,2, \ldots, s-1\}$ be the set  of all the possible values of $\gamma$. This set can be partitioned into: $R_1=\{2m-1: 1\leq m \leq \lfloor \frac{s+3}{4} \rfloor\}$, $R_2=\{2m: 1\leq m \leq \lfloor \frac{s+3}{4} \rfloor\}$, $R_3=\{s-2m+2: 2\leq m \leq \lceil \frac{s-1}{4} \rceil\}$, and $R_4=\{s-2m+1: 1\leq m \leq \lceil \frac{s-5}{4} \rceil\}$. Note that $R_3 = R_4= \varnothing $ when $s\leq 5$. Moreover, $R=R_1\cup R_2 \cup R_3 \cup R_4.$ In fact, when $\gamma$ is odd (i.e., $\gamma\in \{R_1, R_3\}$), we have $R_1=\{1,3,5,\ldots, 2\lfloor \frac{s+3}{4} \rfloor -1\}$ and $R_3=\{s-2\lceil \frac{s-1}{4} \rceil +2, \ldots, s-2\}$. It is easy to verify that, when $s$ is odd, we obtain $2\lfloor \frac{s+3}{4} \rfloor -1+2=s-2\lceil \frac{s-1}{4} \rceil +2$. Similarly, when $\gamma$ is even, we have $2\lfloor \frac{s+3}{4} \rfloor +2=s-2\lceil \frac{s-5}{4} \rceil +1$.

\vskip 0.2cm
\quad\textbf{Case 1.} $\gamma \in R_1$ \\Since $\gamma$ is odd, we have $\min( \lceil \frac{\gamma}{2} \rceil , \lceil \frac{s-\gamma+1}{2} \rceil)=\min(  \frac{\gamma+1}{2} ,  \frac{s-\gamma+2}{2})= \begin{cases} 
\frac{\gamma+1}{2} & \mbox{ if } \ \gamma  \leq  \frac{s+1}{2},\\
\frac{s-\gamma+2}{2} & \mbox{ otherwise. } 
\end{cases}$ \\However, $\gamma=2m-1,$ where $1\leq m \leq \lfloor \frac{s+3}{4} \rfloor$. Thus, if $s\equiv 1 \pmod 4$, then  $\gamma \leq \frac{s+1}{2}$.  Otherwise, if $s\equiv 3 \pmod 4$, then  $\gamma \leq \frac{s-1}{2}$.  Hence, $\min(  \frac{\gamma +1}{2} ,  \frac{s-\gamma+2}{2})= \frac{\gamma +1}{2}.$ Moreover,  $\lambda$ is odd  and $\lambda>\gamma \geq 1$. Thus, $\lambda\geq 3.$  Therefore, when $\gamma\in R_1$, $d= \frac{\lambda +1}{2}  + \frac{s-1}{2} -\frac{\gamma +1}{2}+1.$ 

\vskip 0.2cm
\qquad\quad\textbf{Case 1.1.} $\gamma=1$ \\In this case, $d= \frac{\lambda +s}{2}.$ Let $i=\frac{\lambda-1}{2} s+\frac{s+1}{2}$ be a vertex of $C_n(1,s).$  We have, 

\begin{itemize}
\item $\ell^1(i)=\ell^2(i)=\frac{\lambda-1}{2} + \frac{s+1}{2}= d;$ 
\end{itemize} Let $1\leq t \leq \frac{s}{gcd(n,s)}$ be an integer.  We have $tn+i=(t\lambda + \frac{\lambda-1}{2} )s +t+\frac{s+1}{2}.$ If $t+\frac{s+1}{2} < s$ then,
\begin{itemize}
\item $\ell^{1,t}(i)=t(\lambda + 1)+\frac{\lambda-1}{2} + \frac{s+1}{2} > d$;
\item $\ell^{2,t}(i)=t(\lambda - 1)+\frac{\lambda-1}{2} + \frac{s+1}{2} > d$;
\end{itemize} If $t+\frac{s+1}{2} \geq s,$ i.e., $t \geq \frac{s-1}{2}$, then $tn+i=(t\lambda + \frac{\lambda-1}{2}+\lambda')s +\gamma'$ where $\lambda'\geq 1$, $\gamma' \geq 0$ and $t+\frac{s+1}{2}=\lambda's+\gamma'$.
\begin{itemize}
\item $\ell^{1,t}(i)=t\lambda + \frac{\lambda-1}{2} + \lambda' + \gamma' \geq t\lambda+ \frac{\lambda+1}{2}  > t+ \frac{\lambda+1}{2}.$ So, $\ell^{1,t}(i) > \frac{s-1}{2}+ \frac{\lambda+1}{2}.$ Thus, $\ell^{1,t}(i) > d;$ 
\item $\ell^{2,t}(i) =1+s-\gamma' + t\lambda + \frac{\lambda-1}{2}+\lambda' \geq t\lambda + \frac{\lambda+1}{2}+1.$ Thus,   $\ell^{2,t}(i)> d;$
\end{itemize} We have $tn-i=((t-1)\lambda + \frac{\lambda+1}{2} )s +t-\frac{s+1}{2}.$ 
 If $t-\frac{s+1}{2} < 0,$ then $tn-i=((t-1)\lambda + \frac{\lambda-1}{2})s +t+\frac{s-1}{2},$ 
\begin{itemize}
\item $\ell^{3,t}(i)= (t-1)(\lambda+1)+ \frac{\lambda+1}{2} +\frac{s-1}{2} \geq \frac{\lambda+1}{2} +\frac{s-1}{2} \geq d$;
\end{itemize}
\begin{itemize} 
\item $\ell^{4,t}(i)=(t-1)(\lambda-1)+ \frac{\lambda-1}{2} +\frac{s+1}{2} \geq  \frac{\lambda+1}{2} +\frac{s-1}{2}\geq d$;
\end{itemize} If $0\leq t-\frac{s+1}{2} <s$, i.e., $t\geq \frac{s+1}{2}$, then $tn-i=((t-1)\lambda + \frac{\lambda+1}{2} )s +t-\frac{s+1}{2}.$
\begin{itemize}
\item $\ell^{3,t}(i)=(t-1)\lambda + \frac{\lambda+1}{2} +t-\frac{s+1}{2}\geq (t-1)\lambda + \frac{\lambda+1}{2} >t + \frac{\lambda-1}{2}.$ So,  $\ell^{3,t}(i) >  d.$
\end{itemize}
\begin{itemize} 
\item $\ell^{4,t}(i)=1+s-(t-\frac{s+1}{2})+(t-1)\lambda + \frac{\lambda+1}{2}\geq (t-1)\lambda + \frac{\lambda+1}{2}+1.$ Thus,  $\ell^{4,t}(i) > d.$
\end{itemize} Thus, by Lemma \ref{dis1}, we have $d(i)  =\ell^1 (i)=\ell^2 (i)=d$.

\vskip 0.2cm
\qquad\quad\textbf{Case 1.2.} $\gamma\geq 3$

If $s=5$ then, we have $\gamma=3,$ $\lambda\geq 5$, and $d= \frac{\lambda +1}{2}  + 1.$ Let $i=(\frac{\lambda+1}{2}+1)s$ be a vertex of $C_n(1,s)$.  By a simple calculation of the lengths $\ell^1(i)$ to $\ell^{4,t}(i)$ represented in Lemma \ref{path}, and by applying Lemma \ref{dis1}, we obtain $d(i)  =\ell^1 (i)=d$.

If  $s>5$ then, $\gamma\geq 3$, $\lambda \geq 5$, and $d= \frac{\lambda +1}{2}  + \frac{s+1}{2} -\frac{\gamma +1}{2}.$  Let  $i=\frac{\lambda +1}{2}s+ \frac{\gamma +1}{2} + \frac{s+1}{2}$ be a vertex of $C_n(1,s)$.  Since $s>5$ and $\gamma \leq \frac{s+1}{2}$, we obtain $\frac{\gamma +1}{2} + \frac{s+1}{2}<s$. Note that $\frac{\gamma-1}{2}-1\geq -\frac{\gamma+1}{2}+2$ because $\gamma\geq 3$.  By a simple calculation of the lengths $\ell^1(i)$ to $\ell^{4,t}(i)$ represented in Lemma \ref{path}, and by applying Lemma \ref{dis1}, we obtain $d(i)  =\ell^2 (i)=d$.

\vskip 0.2cm
\quad\textbf{Case 2.} $\gamma \in R_2$ \\ Since $\gamma$ is even, we have $\min( \lceil \frac{\gamma}{2} \rceil , \lceil \frac{s-\gamma+1}{2} \rceil)=\min(  \frac{\gamma}{2} ,  \frac{s-\gamma+1}{2})= \begin{cases} 
\frac{\gamma}{2} & \mbox{ if } \ \gamma  \leq  \frac{s+1}{2},\\
\frac{s-\gamma+1}{2} & \mbox{ otherwise. }
\end{cases}$  \\Moreover, $\gamma =2m$ where $1\leq m \leq \lfloor \frac{s+3}{4} \rfloor$. Thus, for $s$ odd, we obtain  $\gamma \leq \frac{s+3}{2}$.  Hence, we have $\gamma\geq 2$, $\lambda\geq 4$, $s\geq 3$, and $d= \begin{cases} 
\frac{\lambda}{2}  + \frac{s+1}{2} -\frac{s-\gamma+1}{2} & \mbox{ if } \ \gamma  =  \frac{s+3}{2},\\
\frac{\lambda}{2}  + \frac{s+1}{2}-\frac{\gamma}{2} & \mbox{ if } \ \gamma  \leq  \frac{s+1}{2}.
\end{cases}$

\qquad\quad\textbf{Case 2.1.} $\gamma=\frac{s+3}{2}$ \\ In this case, $d=\frac{\lambda}{2}+\frac{\gamma}{2}$. Let $i=\frac{\lambda}{2}s+\frac{\gamma}{2}$ be a vertex of $C_n(1,s)$. Note that since $s>\gamma$, we obtain $s-\frac{\gamma}{2}>\frac{\gamma}{2} $. By a simple calculation of the lengths $\ell^1(i)$ to $\ell^{4,t}(i)$ represented in Lemma \ref{path}, and by applying Lemma \ref{dis1}, we obtain $d(i)  =\ell^1 (i)=d$.

\vskip 0.2cm
\qquad\quad\textbf{Case 2.2.} $\gamma\leq \frac{s+1}{2}$

If $s=3$ then,  $\gamma=2$, $\lambda \geq 4$, and $d=\frac{\lambda}{2}+1.$ Let $i=(\frac{\lambda}{2}+1)s$ be a vertex of $C_n(1,s)$.  By a simple calculation of the lengths $\ell^1(i)$ to $\ell^{4,t}(i)$ represented in Lemma \ref{path}, and by applying Lemma \ref{dis1}, we obtain $d(i)  =\ell^1 (i)=d$.

If $s>3$ then, $\gamma\geq 2$, $\lambda\geq 4$, and $d= \frac{\lambda}{2}  + \frac{s+1}{2} -\frac{\gamma}{2}$. Let $i=\frac{\lambda}{2}s+ \frac{s+1}{2}+\frac{\gamma}{2}$ be a vertex of $C_n(1,s)$. Since $s>3$ and $\gamma\leq \frac{s+1}{2}$, we have $\frac{s+1}{2}+\frac{\gamma}{2}<s.$ Note that $\frac{\gamma}{2}-1 \geq -\frac{\gamma}{2}+1$ because $\gamma\geq 2$.  By a simple calculation of the lengths $\ell^1(i)$ to $\ell^{4,t}(i)$ represented in Lemma \ref{path}, and by applying Lemma \ref{dis1}, we obtain $d(i)  =\ell^2 (i)=d$.

\vskip 0.2cm 
\quad\textbf{Case 3.} $\gamma \in R_3$  \\  In this case, $\gamma=s-2m+2$ where $2\leq m \leq \lceil \frac{s-1}{4} \rceil$. If $s\equiv 1 \pmod 4$, then $\gamma\geq \frac{s+5}{2}.$ Otherwise, if  $s\equiv 3 \pmod 4$, then $\gamma\geq \frac{s+3}{2}.$  Since $\gamma$ is odd and $\gamma \geq \frac{s+1}{2}$, we get $\min( \lceil \frac{\gamma}{2} \rceil , \lceil \frac{s-\gamma+1}{2} \rceil)=\min(  \frac{\gamma+1}{2} ,  \frac{s-\gamma+2}{2} )=\frac{s-\gamma+2}{2}.$ Therefore, $\gamma\geq 3$, $\lambda \geq 5$, $s\geq 7$ (when $s\leq 5, R_3=\emptyset$), and $d=\frac{\lambda +1}{2}  + \frac{s+1}{2} -\frac{s-\gamma+2}{2}.$ Let $i=(\frac{\lambda +1}{2}- \frac{s-\gamma+2}{2})s + \frac{s+1}{2}$ be a vertex $C_n(1,s)$.  By a simple calculation of the lengths $\ell^1(i)$ to $\ell^{4,t}(i)$ represented in Lemma \ref{path}, and by applying Lemma \ref{dis1}, we obtain $d(i)  =\ell^1 (i)=\ell^2 (i)=d$.

\vskip 0.2cm
\quad\textbf{Case 4.} $\gamma \in R_4$  \\ Here we have $\gamma = s-2m+1$ where $1\leq m \leq \lceil \frac{s-5}{4} \rceil$. If $s\equiv 1 \pmod 4$, then $\gamma\geq \frac{s+7}{2}.$ Otherwise, if  $s\equiv 3 \pmod 4$, then $\gamma\geq \frac{s+5}{2}.$ Moreover, since $\gamma$ is even and $\gamma \geq \frac{s+1}{2}$, we have $\min( \lceil \frac{\gamma}{2} \rceil , \lceil \frac{s-\gamma+1}{2} \rceil)=\min(  \frac{\gamma}{2} ,  \frac{s-\gamma+1}{2} )=\frac{s-\gamma+1}{2}.$ Therefore, $\gamma\geq 4$, $\lambda\geq 6$, $s\geq 7$ (when $s\leq 5, R_3=\emptyset$), and   $d=\frac{\lambda }{2}  + \frac{s+1}{2} -\frac{s-\gamma+1}{2}.$ Let $i=(\frac{\lambda }{2}- \frac{s-\gamma+1}{2}+1)s + \frac{s-1}{2}$ be a vertex $C_n(1,s)$.   By a simple calculation of the lengths $\ell^1(i)$ to $\ell^{4,t}(i)$ represented in Lemma \ref{path}, and by applying Lemma \ref{dis1}, we obtain $d(i)  =\ell^1 (i)=d$. 

Consequently,   for all $\gamma\in R$, there exists a vertex $i$ of $C_{n}(1,s)$ such that, 
$d(i)= \lceil \frac{\lambda}{2} \rceil + \frac{s-1}{2} -(\min( \lceil \frac{\gamma}{2} \rceil , \lceil \frac{s-\gamma+1}{2} \rceil) -1).$
\end{proof}

Thus, the following theorem follows from Lemmas \ref{1.1} and \ref{1.2}.

\begin{theorem}
Let $n$ be even and $s$ be odd. We have, 
$$diam(C_{n}(1,s)) = \lceil \frac{\lambda}{2} \rceil + \frac{s-1}{2} -(\min( \lceil \frac{\gamma}{2} \rceil , \lceil \frac{s-\gamma+1}{2} \rceil) -1).$$
\end{theorem}

\subsection{Diameter of $C_{n}(1,s)$ when $n$ is even, $s$ is even, and $\lambda > \gamma > 0$}

\begin{lemma}\cite{chen2005diameter}
\label{2.1}
Let $n$ and $s$ be  even.   We have,

$$diam(C_{n}(1,s))\leq\begin{cases} 
 \lceil \frac{\lambda}{2} \rceil + \frac{s-\gamma}{2} & \mbox{ if } \ \gamma  \leq 2\lceil \frac{s-2}{4} \rceil, \\
\lfloor \frac{\lambda}{2} \rfloor +\frac{\gamma}{2}& \mbox{ otherwise. } 
 \end{cases}$$
\end{lemma}

\begin{lemma}
\label{2.2}
Let $n$ and $s$ be  even. There exists a vertex $i$ of $C_{n}(1,s)$ such that, 
$$d(i)= \begin{cases} 
 \lceil \frac{\lambda}{2} \rceil + \frac{s-\gamma}{2} & \mbox{ if } \ \gamma  \leq 2\lceil \frac{s-2}{4} \rceil, \\
\lfloor \frac{\lambda}{2} \rfloor +\frac{\gamma}{2}& \mbox{ otherwise. } 
 \end{cases}$$
\end{lemma}

\begin{proof}
  Let $d=\begin{cases} 
 \lceil \frac{\lambda}{2} \rceil + \frac{s-\gamma}{2}  & \mbox{ if } \ \gamma  \leq 2\lceil \frac{s-2}{4} \rceil, \\
\lfloor \frac{\lambda}{2} \rfloor +\frac{\gamma}{2}  & \mbox{ otherwise. } 
 \end{cases}$\\In this case, $\gamma$ is even and $s\geq 4$ because  if $s=2$, then $\gamma=0$ (see Theorem \ref{0}).

 \textbf{Case 1.} $\gamma  \leq 2\lceil \frac{s-2}{4} \rceil$ 
 
 \qquad \textbf{Case 1.1.} $\lambda$ is even \\ In this case, $\gamma\geq 2$, $\lambda\geq 4$, $s\geq 4$, and $d=\frac{\lambda}{2}+\frac{s-\gamma}{2}.$ Let $i=(\frac{\lambda}{2}-\frac{\gamma}{2})s+\frac{s}{2}$ be a vertex of $C_n(1,s)$. Table \ref{tab10} represents a simple calculation of the lengths $\ell^1(i)$ to $\ell^{4,t}(i)$ provided in Lemma \ref{path}.

 \begin{table}[H]
\centering
   {\begin{tabular}{ | c || l || r | }
     \hline
     $\ell^1(i)$ & $=\frac{\lambda}{2}+\frac{s-\gamma}{2}$ & $=d$ \\ \hline
     $\ell^2(i)$ & $=\frac{\lambda}{2}+\frac{s-\gamma}{2}+1$ & $>d$ \\ \hline
     $\ell^{1,t}(i), \ell^{2,t}(i)$ & $ >\begin{cases} 
\frac{\lambda}{2}+\frac{s-\gamma}{2}  & \mbox{ if } \ t\gamma+\frac{s}{2}<s,\\
\frac{\lambda}{2}+\frac{s-\gamma}{2}+1 & \mbox{ if } \ t\gamma+\frac{s}{2}\geq s.
\end{cases}$ & $>d$ \\ \hline
     $\ell^{3,t}(i)$ & $  >\begin{cases} 
\frac{\lambda}{2}+\frac{s+\gamma}{2} & \mbox{ if } \  t\gamma-\frac{s}{2}<0,\\
\frac{\lambda}{2}+\frac{s-\gamma}{2} & \mbox{ if } \ 0\leq t\gamma-\frac{s}{2}<s,\\
\frac{\lambda}{2}+\frac{s+\gamma}{2}+1 & \mbox{ if } \ t\gamma-\frac{s}{2} \geq s.
\end{cases}$ & $> d$ \\ \hline
     $\ell^{4,t}(i)$ & $ \begin{cases} 
\geq \frac{\lambda}{2}+\frac{s-\gamma}{2} & \mbox{ if } \  t\gamma-\frac{s}{2}<0,\\
>\frac{\lambda}{2}+\frac{s-\gamma}{2}+1 & \mbox{ if } \ 0\leq t\gamma-\frac{s}{2}<s,\\
> \frac{\lambda}{2}+\frac{s+\gamma}{2}+2 & \mbox{ if } \ t\gamma-\frac{s}{2} \geq s.
\end{cases}$ & $\geq d$ \\\hline
   \end{tabular}}
   \caption{Values of $\ell^1(i)$, $\ell^2(i)$, $\ell^{1,t}(i)$, \ldots, $\ell^{4,t}(i)$.}
\label{tab10}
\end{table}Thus, by Lemma \ref{dis1}, we have $d(i)  =\ell^1 (i)=d$.

 
  \qquad \textbf{Case 1.2.} $\lambda$ is odd

If $\gamma=2$ then,     $\lambda\geq 3$, $s\geq 4$, and $d=\frac{\lambda+1}{2}+\frac{s-2}{2}.$ Let $i=\frac{\lambda-1}{2}s+\frac{s}{2}+1$ be a vertex of $C_n(1,s)$. Note that since $s\geq 4$, we have $\frac{s}{2}+1<s$.  By a simple calculation of the lengths $\ell^1(i)$ to $\ell^{4,t}(i)$ represented in Lemma \ref{path}, and by applying Lemma \ref{dis1}, we obtain $d(i)  =\ell^2 (i)=d$.

If $\gamma\geq 4$ then,   $\lambda\geq 5$, $s\geq 6$, and $d=\frac{\lambda+1}{2}+\frac{s-\gamma}{2}.$ Let $i=\frac{\lambda+1}{2}s+\frac{\gamma}{2}+ \frac{s}{2}+1$ be a vertex of $C_n(1,s)$. Since $s> 4$ and $\gamma\leq \lceil \frac{s-2}{2} \rceil$, we have $\frac{\gamma}{2}+ \frac{s}{2}+1<s$. Note that $\frac{\gamma}{2}-2 > -\frac{\gamma}{2}+1$ because $\gamma\geq 4$.  By a simple calculation of the lengths $\ell^1(i)$ to $\ell^{4,t}(i)$ represented in Lemma \ref{path}, and by applying Lemma \ref{dis1}, we obtain $d(i)  =\ell^2 (i)=d$.

 \textbf{Case 2.} $\gamma  > 2\lceil \frac{s-2}{4} \rceil$

If $\lambda$ is even then, $\gamma\geq 2$, $\lambda\geq 4$, $s\geq 4$, and $d=\frac{\lambda}{2}+\frac{\gamma}{2}.$ Let $i=\frac{n}{2}$ be a vertex of $C_n(1,s)$. Note that since $s>\gamma$, we have $s-\frac{\gamma}{2}>\frac{\gamma}{2}$.  By a simple calculation of the lengths $\ell^1(i)$ to $\ell^{4,t}(i)$ represented in Lemma \ref{path}, and by applying Lemma \ref{dis1}, we obtain $d(i)  =\ell^1 (i)=d$.

If $\lambda$ is odd then,  $\gamma\geq 2$, $\lambda\geq 3$, $s\geq 4$, and $d=\frac{\lambda-1}{2}+\frac{\gamma}{2}.$ Let $i=\frac{\lambda-1}{2}s+\frac{\gamma}{2}$ be a vertex of $C_n(1,s)$. Since $s>\gamma$, we have $s-\frac{\gamma}{2}>\frac{\gamma}{2}$.  By a simple calculation of the lengths $\ell^1(i)$ to $\ell^{4,t}(i)$ represented in Lemma \ref{path}, and by applying Lemma \ref{dis1}, we obtain $d(i)  =\ell^1 (i)=d$. This completes the proof.
\end{proof}

Thus, the following theorem follows from Lemmas \ref{2.1} and \ref{2.2}.

\begin{theorem}
Let $n$ and $s$ be  even.   We have,

$$diam(C_{n}(1,s))=\begin{cases} 
 \lceil \frac{\lambda}{2} \rceil + \frac{s-\gamma}{2} & \mbox{ if } \ \gamma  \leq 2\lceil \frac{s-2}{4} \rceil, \\
\lfloor \frac{\lambda}{2} \rfloor +\frac{\gamma}{2}& \mbox{ otherwise. } 
 \end{cases}$$
\end{theorem}

\subsection{Diameter of $C_{n}(1,s)$ when $n$ is odd, $s$ is odd, and $\lambda > \gamma > 0$}

\begin{lemma}\cite{chen2005diameter}
\label{3.1}
Let $n$ and $s$ be  odd.   We have,

 $$diam(C_{n}(1,s))\leq \lceil \frac{\lambda}{2} \rceil + \frac{s-1}{2} -(\min( \lceil \frac{\gamma +1}{2} \rceil , \lceil \frac{s-\gamma+2}{2} \rceil) -1).$$

\end{lemma}

\begin{lemma}
\label{3.2}
Let $n$ and $s$ be  odd. There exists a vertex $i$ of $C_{n}(1,s)$ such that, 
$$d(i)= \lceil \frac{\lambda}{2} \rceil + \frac{s-1}{2} -(\min( \lceil \frac{\gamma +1}{2} \rceil , \lceil \frac{s-\gamma+2}{2} \rceil) -1).$$
\end{lemma}

\begin{proof}
Let  $d=\lceil \frac{\lambda}{2} \rceil + \frac{s-1}{2} -(\min( \lceil \frac{\gamma +1}{2} \rceil , \lceil \frac{s-\gamma+2}{2} \rceil) -1).$   Let $R=\{1,2, \ldots, s-1\}$ be the set of all the possible values of $\gamma$. This set can be partitioned into: $R_1=\{2m-1: 1\leq m \leq \lfloor \frac{s+3}{4} \rfloor\}$, $R_2=\{2m-2: 2\leq m \leq \lceil \frac{s+5}{4} \rceil\}$, $R_3=\{s-2m+2: 2\leq m \leq \lceil \frac{s-1}{4} \rceil\}$, and $R_4=\{s-2m+3: 2\leq m \leq \lceil \frac{s-1}{4} \rceil\}$. Moreover, $R=R_1\cup R_2 \cup R_3 \cup R_4.$ In fact, when $\gamma$ is odd (i.e., $\gamma\in \{R_1, R_3\}$), we have $R_1=\{1,3,5,\ldots, 2\lfloor \frac{s+3}{4} \rfloor -1\}$ and $R_3=\{s-2\lceil \frac{s-1}{4} \rceil +2, \ldots, s-2\}$. It is easy to verify that, when $s$ is odd, we obtain $2\lfloor \frac{s+3}{4} \rfloor -1+2=s-2\lceil \frac{s-1}{4} \rceil +2$. Similarly, when $\gamma$ is even, we have $2\lceil \frac{s+5}{4} \rceil -2 +2=s-2\lceil \frac{s-1}{4} \rceil +3$.

\vskip 0.2cm
\quad\textbf{Case 1.} $\gamma \in R_1$ \\ Since $\gamma$ is odd, we have $\min( \lceil \frac{\gamma +1}{2} \rceil , \lceil \frac{s-\gamma+2}{2} \rceil)=\min(  \frac{\gamma+1}{2} ,  \frac{s-\gamma+2}{2})= \begin{cases} 
\frac{\gamma+1}{2} & \mbox{ if } \ \gamma  \leq  \frac{s+1}{2},\\
\frac{s-\gamma+2}{2} & \mbox{ otherwise. } 
\end{cases}$ \\However, we have $\gamma =2m-1$ where $1\leq m \leq \lfloor \frac{s+3}{4} \rfloor$. Thus, if $s\equiv 1 \pmod 4$, then  $\gamma \leq \frac{s+1}{2}$.  Otherwise, if $s\equiv 3 \pmod 4$, then  $\gamma \leq \frac{s-1}{2}$.  Hence, $\min(  \frac{\gamma +1}{2} ,  \frac{s-\gamma+2}{2})= \frac{\gamma +1}{2}.$   Moreover, we have $\gamma\geq 1$, $\lambda\geq 2$, $s\geq 3$, and  $d= \frac{\lambda}{2}  + \frac{s+1}{2} -\frac{\gamma +1}{2}.$  Let $i=(\frac{\lambda}{2}-1)s+\frac{s-1}{2}+\frac{\gamma +1}{2}$ be a vertex of $C_n(1,s).$  Since $\gamma \leq \frac{s-1}{2}$, we have $\frac{s-1}{2}+\frac{\gamma +1}{2}<s$. Note that $\frac{\gamma+1}{2}-1\geq -\frac{\gamma+1}{2}+1$ because $\gamma\geq 1$.  Table \ref{tab15} represents a simple calculation of the lengths $\ell^1(i)$ to $\ell^{4,t}(i)$ provided in Lemma \ref{path}.

\begin{table}[H]
\centering
   {\begin{tabular}{ | c || l || r | }
     \hline
     $\ell^1(i)$ & $=\frac{\lambda}{2}+\frac{s-1}{2}+\frac{\gamma +1}{2}-1$ & $\geq d$ \\ \hline
     $\ell^2(i)$ & $=\frac{\lambda}{2}  + \frac{s+1}{2} -\frac{\gamma +1}{2}$ & $=d$ \\ \hline
     $\ell^{1,t}(i), \ell^{2,t}(i)$ & $ >\begin{cases} 
\frac{\lambda}{2}  + \frac{s+1}{2} -\frac{\gamma +1}{2} & \mbox{ if }  t\gamma+\frac{\gamma+1}{2}+\frac{s-1}{2}<s,\\
\frac{\lambda}{2}  + \frac{s+1}{2} -\frac{\gamma +1}{2} & \mbox{ if }  t\gamma+\frac{\gamma+1}{2}+\frac{s-1}{2}\geq s.
\end{cases}$ & $>d$ \\ \hline
     $\ell^{3,t}(i), \ell^{4,t}(i)$ & $  >\begin{cases} 
\frac{\lambda}{2}  + \frac{s+1}{2} -\frac{\gamma +1}{2} & \mbox{ if } \ t\gamma-\frac{\gamma+1}{2}-\frac{s-1}{2}<0,\\
\frac{\lambda}{2}  + \frac{s+1}{2} -\frac{\gamma +1}{2} & \mbox{ if } \ t\gamma-\frac{\gamma+1}{2}-\frac{s-1}{2}<s,\\
\frac{\lambda}{2}  + \frac{s+1}{2} -\frac{\gamma +1}{2} & \mbox{ if }  t\gamma-\frac{\gamma+1}{2}-\frac{s-1}{2}\geq s.
\end{cases}$ & $> d$ \\ \hline
   \end{tabular}}
   \caption{Values of $\ell^1(i)$, $\ell^2(i)$, $\ell^{1,t}(i)$, \ldots, $\ell^{4,t}(i)$.}
\label{tab15}
\end{table} Thus, by Lemma \ref{dis1},  we have $d(i)  =\ell^2 (i)=d$.

\vskip 0.1cm 
\quad\textbf{Case 2.} $\gamma \in R_2$ \\ Since $\gamma$ is even, we have $\min( \lceil \frac{\gamma +1}{2} \rceil , \lceil \frac{s-\gamma+2}{2} \rceil)=\min(  \frac{\gamma+2}{2} ,  \frac{s-\gamma+3}{2})= \begin{cases} 
\frac{\gamma+2}{2} & \mbox{ if } \ \gamma  \leq  \frac{s+1}{2},\\
\frac{s-\gamma+3}{2} & \mbox{ if } \gamma  \geq  \frac{s+3}{2}.
\end{cases}$ \\However, we have $\gamma=2m-2$ where $ 2\leq m \leq \lceil \frac{s+5}{4} \rceil.$ Thus, if $s\equiv 1 \pmod 4$, then  $\gamma \leq \frac{s+3}{2}$.  Otherwise, if $s\equiv 3 \pmod 4$, then  $\gamma \leq \frac{s+1}{2}$.  Hence, $\min(  \frac{\gamma +1}{2} ,  \frac{s-\gamma+2}{2})= \frac{\gamma +1}{2}.$  Therefore,  $\gamma\geq 2$, $\lambda\geq 3$, $s\geq 3$, and $d= \begin{cases} 
\frac{\lambda+1}{2}  + \frac{s+1}{2} -\frac{s-\gamma+3}{2} & \mbox{ if } \ \gamma  =  \frac{s+3}{2},\\
\frac{\lambda+1}{2}  + \frac{s+1}{2}-\frac{\gamma+2}{2} & \mbox{ if } \ \gamma  \leq  \frac{s+1}{2}.
\end{cases}$

\qquad\quad\textbf{Case 2.1.} $\gamma=\frac{s+3}{2}$ \\ In this case, $d=\frac{\lambda+1}{2}+\frac{s-1}{4}$. Let $i=\frac{\lambda+1}{2}s+\frac{s-1}{4}$ be a vertex of $C_n(1,s)$. Note that  $s-\frac{s-1}{4}>\frac{s-1}{4} $.  By a simple calculation of the lengths $\ell^1(i)$ to $\ell^{4,t}(i)$ represented in Lemma \ref{path}, and by applying Lemma \ref{dis1}, we obtain $d(i)  =\ell^1 (i)=d$.

\qquad\quad\textbf{Case 2.2.} $\gamma\leq \frac{s+1}{2}$

If $s=3$ then,   $\gamma=2$, $\lambda \geq 3$, and $d=\frac{\lambda+1}{2}.$ Let $i=\frac{\lambda+1}{2}s$ be a vertex of $C_n(1,s)$.  By a simple calculation of the lengths $\ell^1(i)$ to $\ell^{4,t}(i)$ represented in Lemma \ref{path}, and by applying Lemma \ref{dis1}, we obtain $d(i)  =\ell^1 (i)=d$.

If $s>3$ then, $\gamma\geq 2$, $\lambda\geq 3$, and $d= \frac{\lambda+1}{2}  + \frac{s+1}{2} -\frac{\gamma+2}{2}$. Let $i=\frac{\lambda-1}{2}s+ \frac{s-1}{2}+\frac{\gamma+2}{2}$ be a vertex of $C_n(1,s)$. Since $s>3$ and $\gamma\leq \frac{s+1}{2}$, we get $\frac{s-1}{2}+\frac{\gamma+2}{2}<s.$ Note that $\frac{\gamma+2}{2}>  -\frac{\gamma+2}{2}+2$ and $-\frac{\gamma-2}{2}>  -\frac{\gamma+2}{2}+1$ because $\gamma\geq 2$.  By a simple calculation of the lengths $\ell^1(i)$ to $\ell^{4,t}(i)$ represented in Lemma \ref{path}, and by applying Lemma \ref{dis1}, we obtain $d(i)  =\ell^2 (i)=d$.

\vskip 0.1cm
\quad\textbf{Case 3.} $\gamma \in R_3$  \\ In this case,  $\gamma\geq 3$, $\lambda \geq 4$, $s\geq 3$, and $d=\frac{\lambda }{2}  + \frac{s+1}{2} -\frac{s-\gamma+2}{2}.$ Let $i=(\frac{\lambda}{2}- \frac{s-\gamma+2}{2})s + \frac{s+1}{2}$ be a vertex of $C_n(1,s)$.  By a simple calculation of the lengths $\ell^1(i)$ to $\ell^{4,t}(i)$ represented in Lemma \ref{path}, and by applying Lemma \ref{dis1}, we obtain $d(i)  =\ell^1 (i)=\ell^2 (i)=d$.

\vskip 0.1cm
\quad\textbf{Case 4.} $\gamma \in R_4$  \\ We have $\gamma\geq 4$, $\lambda\geq 5$, $s\geq 3,$ and   $d=\frac{\lambda +1}{2}  + \frac{s+1}{2} -\frac{s-\gamma+3}{2}.$ Let $i=(\frac{\lambda+1}{2}- \frac{s-\gamma+3}{2})s + \frac{s-1}{2}$ be a vertex of $C_n(1,s)$.  By a simple calculation of the lengths $\ell^1(i)$ to $\ell^{4,t}(i)$ represented in Lemma \ref{path}, and by applying Lemma \ref{dis1}, we obtain $d(i)  =\ell^1 (i)=\ell^2 (i)=d$. \\Consequently,   for all $\gamma\in R$, there exists a vertex $i$ of $C_{n}(1,s)$ such that, 
$d(i)= \lceil \frac{\lambda}{2} \rceil + \frac{s-1}{2} -(\min( \lceil \frac{\gamma +1}{2} \rceil , \lceil \frac{s-\gamma+2}{2} \rceil) -1).$
\end{proof}

Thus, the following theorem follows from Lemmas \ref{3.1} and \ref{3.2}.

\begin{theorem}
Let $n$ and $s$ be  odd.   We have,

 $$diam(C_{n}(1,s))= \lceil \frac{\lambda}{2} \rceil + \frac{s-1}{2} -(\min( \lceil \frac{\gamma +1}{2} \rceil , \lceil \frac{s-\gamma+2}{2} \rceil) -1).$$
\end{theorem}

\subsection{Diameter of $C_{n}(1,s)$ when $n$ is odd, $s$ is even, and $\lambda > \gamma > 0$}

\begin{lemma}\cite{chen2005diameter}
\label{4.1}
Let $n$ odd and $s$ even.  we have,
$$diam(C_{n}(1,s))\leq \begin{cases} 
\lceil \frac{\lambda}{2} \rceil + \frac{s-2}{2} & \mbox{ if } \ \gamma = 1 \mbox{ or } \gamma = s-1, \\
 \lfloor \frac{\lambda}{2} \rfloor + \frac{s-\gamma+1}{2} & \mbox{ if } \ 3\leq \gamma  \leq 2\lceil \frac{s}{4} \rceil -1, \\
\lceil \frac{\lambda}{2} \rceil +\frac{\gamma -1}{2} & \mbox{ otherwise. } 
 \end{cases}$$ 
\end{lemma}

\begin{lemma}
\label{4.2}
Let $n$ odd and $s$ even. There exists a vertex $i$ of $C_{n}(1,s)$ such that, 
$$d(i)= \begin{cases} 
\lceil \frac{\lambda}{2} \rceil + \frac{s-2}{2} & \mbox{ if } \ \gamma = 1 \mbox{ or } \gamma = s-1, \\
 \lfloor \frac{\lambda}{2} \rfloor + \frac{s-\gamma+1}{2} & \mbox{ if } \ 3\leq \gamma  \leq 2\lceil \frac{s}{4} \rceil -1, \\
\lceil \frac{\lambda}{2} \rceil +\frac{\gamma -1}{2} & \mbox{ otherwise. } 
 \end{cases}$$
\end{lemma}

\begin{proof}
Let     $d=\begin{cases} 
\lceil \frac{\lambda}{2} \rceil + \frac{s-2}{2} & \mbox{ if $\gamma \in \{1,s-1\}$,} \\
 \lfloor \frac{\lambda}{2} \rfloor + \frac{s-\gamma+1}{2} & \mbox{ if } \ 3\leq \gamma  \leq \lceil \frac{s}{2} \rceil -1, \\
\lceil \frac{\lambda}{2} \rceil +\frac{\gamma -1}{2} & \mbox{ otherwise. } 
 \end{cases}$ 

\textbf{Case 1.} $\gamma \in \{1,s-1\}$\\ If $\gamma=1$ then,  $\lambda\geq 2$, $s \geq 4$, and $d=\lceil\frac{\lambda}{2}\rceil+\frac{s-2}{2}.$  Let $i=(\lceil\frac{\lambda}{2}\rceil -1)s+\frac{s}{2}$ be a vertex of $C_n(1,s)$. Table \ref{tab21} represents a simple calculation of the lengths $\ell^1(i)$ to $\ell^{4,t}(i)$ provided in Lemma \ref{path}.

\begin{table}[H]
\centering
   {\begin{tabular}{ | c || l || r | }
     \hline
     $\ell^1(i)$ & $=\lceil\frac{\lambda}{2}\rceil+\frac{s-2}{2}$ & $=d$ \\ \hline
     $\ell^2(i)$ & $=\lceil\frac{\lambda}{2}\rceil+\frac{s}{2}$ & $>d$ \\ \hline
     $\ell^{1,t}(i), \ell^{2,t}(i)$ & $ >\begin{cases} 
\lceil\frac{\lambda}{2}\rceil+\frac{s-2}{2} & \mbox{ if } \ t+\frac{s}{2}<s,\\
\lceil\frac{\lambda}{2}\rceil+\frac{s-2}{2} & \mbox{ if } \ t+\frac{s}{2}\geq s.
\end{cases}$ & $>d$ \\ \hline
     $\ell^{3,t}(i)$ & $ >\begin{cases} 
\lceil\frac{\lambda}{2}\rceil+\frac{s-2}{2} & \mbox{ if } \ t-\frac{s}{2}<0,\\
\lceil\frac{\lambda}{2}\rceil+\frac{s-2}{2} & \mbox{ if } \ 0\leq t-\frac{s}{2}<s.
\end{cases}$ & $> d$ \\ \hline
     $\ell^{4,t}(i)$ & $ \begin{cases} 
\geq \lceil\frac{\lambda}{2}\rceil+\frac{s-2}{2} & \mbox{ if } \ t-\frac{s}{2}<0,\\
>\lceil\frac{\lambda}{2}\rceil+\frac{s-2}{2} & \mbox{ if } \ 0\leq t-\frac{s}{2}<s.
\end{cases}$ & $\geq d$ \\ 
     \hline
   \end{tabular}}
   \caption{Values of $\ell^1(i)$, $\ell^2(i)$, $\ell^{1,t}(i)$, \ldots, $\ell^{4,t}(i)$.}
\label{tab21}
\end{table}Thus, by Lemma \ref{dis1},  we have $d(i)  =\ell^1 (i)=d$.

If $\gamma=s-1$ then,  $\lambda\geq s$, $s \geq 4$, and $d=\lceil\frac{\lambda}{2}\rceil+\frac{s-2}{2}.$  Let      $i=(\lceil\frac{\lambda}{2}\rceil -1)s+\frac{s}{2}$
 be a vertex of $C_n(1,s)$.  By a simple calculation of the lengths $\ell^1(i)$ to $\ell^{4,t}(i)$ represented in Lemma \ref{path}, and by applying Lemma \ref{dis1}, we obtain $d(i)  =\ell^1 (i)=d$.
 
 \textbf{Case 2.} $3\leq \gamma \leq \lceil \frac{s}{2}\rceil -1$

If $\lambda$ is even then,   $\lambda\geq 4$, $s \geq 4$, and $d=\frac{\lambda}{2}+\frac{s}{2}-\frac{\gamma -1}{2}.$ Let $i=(\frac{\lambda}{2}-\frac{\gamma -1}{2})s+\frac{s}{2}+1$ be a vertex of $C_n(1,s)$. Note that since $s\geq 4$, we have $\frac{s}{2}+1<s$.  By a simple calculation of the lengths $\ell^1(i)$ to $\ell^{4,t}(i)$ represented in Lemma \ref{path}, and by applying Lemma \ref{dis1}, we obtain $d(i)  =\ell^2 (i)=d$.

 If $\lambda$ is odd then, $\gamma\geq 3$,  $\lambda\geq 5$, $s \geq 4$, and $d=\frac{\lambda-1}{2}+\frac{s}{2}-\frac{\gamma -1}{2}.$ Let $i=\frac{\lambda-1}{2}s+\frac{\gamma -1}{2}+\frac{s}{2}+1$ be a vertex of $C_n(1,s)$. Note that since $\gamma \leq \lceil \frac{s}{2}\rceil -1$ and $s\geq 4$, we get $\frac{\gamma -1}{2}+\frac{s}{2}+1<s$.  By a simple calculation of the lengths $\ell^1(i)$ to $\ell^{4,t}(i)$ represented in Lemma \ref{path}, and by applying Lemma \ref{dis1}, we obtain $d(i)  =\ell^2 (i)=d$.
 
\textbf{Case 3.} $\lceil \frac{s}{2}\rceil\leq \gamma <s-1  $ \\ We have $\gamma\geq 3$,  $\lambda\geq 4$, $s \geq 4$, and $d=\lceil\frac{\lambda}{2}\rceil+\frac{\gamma -1}{2}.$ Let $i=\lceil\frac{\lambda}{2}\rceil s+\frac{\gamma -1}{2}$ be a vertex of $C_n(1,s)$. Note that since $s>\gamma$, we have $s-\frac{\gamma-1}{2}>\frac{\gamma+1}{2}$.  By a simple calculation of the lengths $\ell^1(i)$ to $\ell^{4,t}(i)$ represented in Lemma \ref{path}, and by applying Lemma \ref{dis1}, we obtain $d(i)  =\ell^1 (i)=d$. This completes the proof.
\end{proof}

Thus, the following theorem follows from Lemmas \ref{4.1} and \ref{4.2}.

\begin{theorem}
Let $n$ odd and $s$ even. We have,
$$diam(C_{n}(1,s))=\begin{cases} 
\lceil \frac{\lambda}{2} \rceil + \frac{s-2}{2} & \mbox{ if } \ \gamma = 1 \mbox{ or } \gamma = s-1, \\
 \lfloor \frac{\lambda}{2} \rfloor + \frac{s-\gamma+1}{2} & \mbox{ if } \ 3\leq \gamma  \leq 2\lceil \frac{s}{4} \rceil -1, \\
\lceil \frac{\lambda}{2} \rceil +\frac{\gamma -1}{2} & \mbox{ otherwise. } 
 \end{cases}$$
\end{theorem}

Next, we discuss the case when $\gamma\neq 0$ and $\lambda \leq \gamma$.

\subsection{Diameter of $C_{n}(1,s)$ when $\lambda \leq \gamma$}

\begin{theorem}\cite{chen2005diameter}
Let $s = a\gamma + b,$ where $b$ $(0 < b < \gamma)$ is the remainder of dividing $s$ by $\gamma$. Define $p_0=\lfloor \frac{\lambda+\gamma}{2} \rfloor$, $p_1=\lfloor \frac{\gamma-b+(a+1)\lambda+1}{2} \rfloor$, $p_2=\lfloor \frac{\gamma+b+(a-1)\lambda+1}{2} \rfloor$, and $p_3=\lfloor \frac{b+a\lambda+1}{2} \rfloor$. Let $e_1=\min\{\max\{p_1, p_3\},\max\{p_0, p_2\}\}.$ If $\lambda\leq \gamma$ and $b\leq a\lambda+1,$ then 

$$diam(C_{n}(1,s))=\begin{cases} 
p_1 -1  & \mbox{ if } p_1=p_2 \mbox{ and } (\gamma+b)(a\lambda-\lambda+1) \equiv 1 \pmod 2, \\
 e_1 & \mbox{ otherwise. } 
 \end{cases}$$
\end{theorem}

For the rest of the values of $n$ and $s$, our algorithm gives an exact value for the diameter of $C_{n}(1,s)$, but does not provide an exact formula. Therefore, we present, in the next section,   upper bounds on the diameter of  $C_n(1,s)$ for all values of $n$ and $s$.

\subsection{Upper bound for $diam(C_n(1,s))$}
\label{Subsec:3}

In 1990, Du \textit{et al.} \cite{du1990combinatorial} gave the following upper bound on the diameter of $C_n(1,s)$:

\begin{equation} \label{1}
diam(C_n(1,s)) \leq \max\{ \lfloor \frac{n}{s} \rfloor+1, n-\lfloor \frac{n}{s} \rfloor s-2, ( \lfloor \frac{n}{s} \rfloor+1)s-n-1\}.
\end{equation}

Another upper bound was given by G{\"o}bel and Neutel \cite{gobel2000cyclic}:

\begin{equation} \label{2} 
 diam(C_n(1,s)) \leq diam(C_n(1,2))=\lfloor \frac{n+2}{4}\rfloor.
 \end{equation}

 Next, we present our upper bound on the diameter of  $C_n(1,s)$ for all values of $n$ and $s$.

\begin{theorem}
\label{prop10}
For all $n$ and $s,$ we have
\begin{equation} \label{3}
diam(C_n(1,s)) \leq \lfloor \frac{\lfloor \frac{n}{2} \rfloor}{s}\rfloor + \lceil \frac{s}{2} \rceil.
\end{equation} 
\end{theorem}

\begin{proof}
 If $i\leq \lfloor \frac{n}{2} \rfloor,$ then $\ell^1(i) \leq r + \lfloor \frac{\lfloor \frac{n}{2} \rfloor}{s} \rfloor  $ and $\ell^2(i) \leq 1+s-r+ \lfloor \frac{\lfloor \frac{n}{2} \rfloor}{s} \rfloor.$ Since, by Lemma \ref{dis1}, $d(i) \leq \min(\ell^1(i), \ell^2(i)),$ then  $d(i) \leq \lfloor \frac{\lfloor \frac{n}{2} \rfloor}{s}\rfloor + \lceil \frac{s}{2} \rceil .$ In fact, $\min(\ell^1(i), \ell^2(i))=\ell^1(i)$ if $r \leq \lceil \frac{s}{2} \rceil$ and we have $\min(\ell^1(i), \ell^2(i))=\ell^1(i) \leq \lceil \frac{s}{2} \rceil+ \lfloor \frac{\lfloor \frac{n}{2} \rfloor}{s} \rfloor  .$ If $r > \lceil \frac{s}{2} \rceil,$ then $\min(\ell^1(i), \ell^2(i))=\ell^2(i) < 1+s-\lceil \frac{s}{2} \rceil +\lfloor \frac{\lfloor \frac{n}{2} \rfloor}{s}\rfloor \leq \lceil \frac{s}{2} \rceil +\lfloor \frac{\lfloor \frac{n}{2} \rfloor}{s}\rfloor.$ \\ Thus, for $i\leq \lfloor \frac{n}{2} \rfloor,$ $d(i) \leq \lfloor \frac{\lfloor \frac{n}{2} \rfloor}{s}\rfloor + \lceil \frac{s}{2} \rceil$. By Lemma \ref{dis2}, we conclude that $diam(C_n(1,s)) \leq \lfloor \frac{\lfloor \frac{n}{2} \rfloor}{s}\rfloor + \lceil \frac{s}{2} \rceil$. 
\end{proof}

By combining (\ref{1}), (\ref{2}) and (\ref{3}), we obtain the following result.

\begin{theorem}
For all $n$ and $s,$ we have
$$diam(C_n(1,s)) \leq \min (\max\{ \lfloor \frac{n}{s} \rfloor+1, n-\lfloor \frac{n}{s} \rfloor s-2, ( \lfloor \frac{n}{s} \rfloor+1)s-n-1\}, \lfloor \frac{n+2}{4}\rfloor, \lfloor \frac{\lfloor \frac{n}{2} \rfloor}{s}\rfloor + \lceil \frac{s}{2} \rceil ).$$ 
\end{theorem}


\paragraph*{\textbf{Conclusion}}

To the best of our knowledge, despite the regularity of circulant graphs, there is no formulas giving exact values for the distance and the diameter of $C_n(1,s)$ for all $n$ and $s$. Our approach which makes it possible to determine the distances in a circulant, allowed us to give exact formulas for the diameter of circulant graphs $C_n(1,s)$, for almost all values of $n$ and $s$, and to provide upper bounds for the rest of the cases of $n$ and $s$.

\bibliographystyle{unsrt}  


\end{document}